\documentclass[12pt,a4paper]{article}

\usepackage{cite}				 
\usepackage{pdfpages}			 
\usepackage{wrapfig}			 
\usepackage{bm} 				 
\usepackage{upgreek}			 
\usepackage{dsfont}				 
\usepackage{simplewick}			 
\usepackage{mathtools}		     
\usepackage{microtype}			 
\usepackage{marvosym}			 
\usepackage{color}				 
\usepackage{transparent}	     
\usepackage{placeins}			 
\usepackage{caption}
\usepackage{subcaption}
\usepackage{mdframed,mdwlist}    
\usepackage{graphicx}            
\usepackage{longtable}           
\usepackage{mathdots}            
\usepackage{eucal}               
\usepackage{array}               
\usepackage{stmaryrd}            
\usepackage{amsthm}              
\usepackage{pifont}              
\usepackage{lipsum}              
\usepackage{enumerate}           
\usepackage[all]{xy}             
\usepackage{amsmath}             
\usepackage{amssymb}             
\usepackage[utf8]{inputenc}      
\usepackage{fancyhdr}            
\usepackage{blindtext}           
\usepackage{tikz}                
\usepackage[figuresright]{rotating}	
\usetikzlibrary{shapes.geometric, arrows}
\usepackage{titlesec}

\usepackage{algorithm}
\usepackage{algpseudocode}
\usepackage{xcolor,colortbl}

\titleformat{\paragraph}
{\normalfont\normalsize\bfseries}{\theparagraph}{1em}{}
\titlespacing*{\paragraph}
{0pt}{3.25ex plus 1ex minus .2ex}{1.5ex plus .2ex}

\newtheorem{theorem}{Theorem}
\newtheorem{corollary}[theorem]{Corollary}
\newtheorem{conjecture}[theorem]{Conjecture}

\newtheorem{observation}[theorem]{Observation}
\newtheorem{lemma}[theorem]{Lemma}
\theoremstyle{definition}

\usetikzlibrary{arrows.meta}
\usetikzlibrary{graphs}
\usetikzlibrary{backgrounds}
\usetikzlibrary{calc}
\usetikzlibrary{positioning}
\usepackage{tkz-graph}
\definecolor{darkyellow}{rgb}{0.960938, 0.742188, 0}
\definecolor{darkred}{rgb}{0.7, 0, 0}
\definecolor{darkgreen}{rgb}{0, 0.5, 0}
\definecolor{darkblue}{rgb}{0, 0, 0.7}
\definecolor{nodepink}{RGB}{254,199,184}
\definecolor{ColourSED1F_Factor1}{RGB}{230, 25, 75}
\definecolor{ColourSED1F_Factor2}{RGB}{60, 180, 75}
\definecolor{ColourSED1F_Factor3}{RGB}{255, 225, 25}
\definecolor{ColourSED1F_Factor4}{RGB}{0, 130, 200}
\definecolor{ColourSED1F_Factor5}{RGB}{245, 130, 48}
\definecolor{ColourSED1F_Factor6}{RGB}{145, 30, 180}
\definecolor{ColourSED1F_Factor7}{RGB}{70, 240, 240}
\definecolor{ColourSED1F_Factor8}{RGB}{240, 50, 230}
\definecolor{ColourSED1F_Factor9}{RGB}{128,0,0}
\definecolor{ColourFactor_Graph1}{RGB}{0, 114, 178}
\definecolor{ColourFactor_Graph2}{RGB}{230, 159, 0}
\definecolor{row_highlight_yellow}{RGB}{255,242,171}

\tikzset{K10ED1F_Factor1/.style={ColourSED1F_Factor1, line width=2pt}}
\tikzset{K10ED1F_Factor2/.style={ColourSED1F_Factor2, line width=2pt}}
\tikzset{K10ED1F_Factor3/.style={ColourSED1F_Factor3, line width=2pt}}
\tikzset{K10ED1F_Factor4/.style={ColourSED1F_Factor4, line width=2pt}}
\tikzset{K10ED1F_Factor5/.style={ColourSED1F_Factor5, line width=2pt}}
\tikzset{K10ED1F_Factor6/.style={ColourSED1F_Factor6, line width=2pt}}
\tikzset{K10ED1F_Factor7/.style={ColourSED1F_Factor7, line width=2pt}}
\tikzset{K10ED1F_Factor8/.style={ColourSED1F_Factor8, line width=2pt}}
\tikzset{K10ED1F_Factor9/.style={ColourSED1F_Factor9, line width=2pt}}
\tikzset{K10ED1F_Factor_Graph1/.style={ColourFactor_Graph1, line width=2pt}}
\tikzset{K10ED1F_Factor_Graph2/.style={ColourFactor_Graph2, line width=2pt}}

\tikzset{F1/.style={darkred, line width=2pt}}
\tikzset{F2/.style={darkgreen, dash pattern=on 8pt off 2pt,  line width=2pt}}
\tikzset{F3/.style={darkblue, dotted,   line width=2pt}}
\tikzset{circF2/.style={darkgreen, dash pattern=on 8pt off 2pt,  line width=1pt}}
\tikzset{circF3/.style={darkred, line width=1pt}}
\tikzset{circF1/.style={darkblue, dotted,   line width=1pt}}
\tikzset{circ_darkyellow/.style={darkyellow,dash pattern=on 8pt off 2pt,  line width=3pt}}
\tikzset{circ_darkred/.style={darkred,     line width=3pt}}
\tikzset{circ_darkblue/.style={darkblue,dotted,    line width=3pt}}
\tikzset{circ_darkgreen/.style={darkgreen, dash dot,  line width=3pt}}

\newcommand{\Mod}[1]{\ (\mathrm{mod}\ #1)}

\usepackage{url}

\usepackage{hyperref}

\usepackage{memhfixc}

\usepackage{amsmath}
\usepackage{amssymb}
\usepackage{amsthm}

\usepackage{setspace}
\allowdisplaybreaks
\usepackage{float}

\usepackage{multicol}
\usepackage{authblk}
\title{Balanced 1-Factorisations of 3- and 4-Regular Circulant Graphs}
\author{Jeremy Mitchell\thanks{jeremy.mitchell@uq.edu.au}}
\affil{{\small School of Mathematics and Physics\\ The University of Queensland\\ QLD 4072, Australia} }
\date{\today}

\begin{document}
\maketitle

\begin{abstract}
    We investigate 1-factorisations in which the 2-regular graphs that occur as the union of a pair of 1-factors appear an equal number of times across the unions of all pairs of 1-factors in the 1-factorisation. We call such 1-factorisations balanced 1-factorisations (B1Fs) and we present some results on B1Fs of 3- and 4-regular circulant graphs.
\end{abstract}

\section{Introduction}
A 1-regular spanning subgraph of a graph is known as a \emph{1-factor}. A partition of the edge set of a graph $G$ into 1-factors is called a \emph{1-factorisation} of $G$.
A natural question is: under what conditions does a 1-factorisation of the complete graph on $n$ vertices, $K_n$, exist? Clearly $n$ must be even, and one of the earliest proofs that this condition is sufficient is Kirkman's 1847 construction of 1-factorisations of $K_n$ for all even integers $n\geq 2$ \cite{kirkman1847problem}.

Given a 1-factorisation of a graph $G$, a well-studied problem is to ask if the 1-factorisation has the property that the union of each pair of 1-factors is isomorphic to the same subgraph $H$ of $G$.  Such a 1-factorisation is called a \emph{uniform 1-factorisation} (U1F) of $G$ and the subgraph $H$ is called the \emph{common graph}. Furthermore, a uniform 1-factorisation in which the common graph is a Hamilton cycle is called a \emph{perfect 1-factorisation} (P1F). The following famous conjecture is due to Kotzig \cite{kotzig1964theory}.

\begin{conjecture}\textup{\cite{kotzig1964theory}}
    For any $n\geq 2$, $K_{2n}$ admits a perfect 1-factorisation.
\end{conjecture}

Kotzig \cite{kotzig1963hamilton} provided an infinite family of 1-factorisations of the complete graph $K_{2n}$ that are perfect when $2n-1$ is an odd prime. Bryant, Maenhaut, and Wanless \cite{bryant2006new} constructed another infinite family of P1Fs of $K_{2n}$ where $2n-1$ is an odd prime, which is not isomorphic to the family given by Kotzig. Anderson \cite{anderson1973finite} gave an infinite family of 1-factorisations of $K_{2n}$ that are perfect when $n$ is an odd prime. Besides these infinite families there are a number of sporadic values of $n$ for which $K_{2n}$ has been shown to admit a P1F. Most recently a P1F of $K_{56}$ was found by Pike \cite{pike2019perfect}, which leaves $K_{64}$ as the smallest complete graph for which the existence of a P1F is unknown; for more information on the orders of complete graphs with known P1Fs, { a paper on the number of non-isomorphic P1Fs of $K_{16}$ by Gill and Wanless \cite{gill2020perfect} is recommended}.

For uniform 1-factorisations that are not perfect, the common graph will be a collection of two or more disjoint cycles of even lengths. We say that a U1F has \emph{type} $(c_1,c_2,\dots,c_t)$ if the common graph of the U1F is a collection of $t$ cycles of lengths $c_1,c_2,\dots,c_t$. For complete graphs $K_{2n}$ with $2n\leq 16$, all types of U1Fs up to isomorphism are known due to a result by Meszka and Rosa \cite{meszka2003perfect}.

\begin{theorem}\textup{\cite{meszka2003perfect}}
    If $\mathcal{F}$ is a U1F of $K_{2n}$, where $2n\leq 16$, then $\mathcal{F}$ is one of the following:
    \begin{enumerate}[(a)]
        \item \label{a} a P1F;
        \item \label{b} a U1F of $K_8$ of type $(4,4)$;
        \item \label{c} a U1F of $K_{10}$ of type $(4,6)$;
        \item \label{d} a U1F of $K_{12}$ of type $(6,6)$;
        \item \label{e} a U1F of $K_{16}$ of type $(4,4,4,4)$.
    \end{enumerate}
    Further, the U1Fs from cases \ref{b}, \ref{c}, \ref{d}, \ref{e} are unique up to isomorphism.
\end{theorem}

In 2013, Herke and Maenhaut \cite{Herke2013P1FCirc} began the study of P1Fs of circulant graphs. The \emph{circulant graph} \(Circ(n,D)\) where \(n\) is a positive integer and \(D\) is a subset of \(\left\lbrace 1,2, \dots, \left\lfloor \frac{n}{2} \right\rfloor \right\rbrace\) is the graph with vertex set \(\mathbb{Z}_{n}\) and where two vertices \(v\) and \(u\) are adjacent if and only if \(|v-u|=d \pmod n\) for some \(d\in D\). A useful fact is that a circulant graph, \(Circ(n,\left\lbrace a,b \right\rbrace)\) is connected if and only if \(\gcd(n,a,b)=1\) (see for example \cite{Barajas_Circ_Chrom_numb}).
Herke and Maenhaut \cite{Herke2013P1FCirc} showed that, up to isomorphism, there are only two connected 3-regular circulant graphs of order $2n$, namely \(Circ(2n, \left\lbrace 1,n \right\rbrace)\) and \(Circ(2n, \left\lbrace 2,n \right\rbrace )\). They also characterised the 3-regular circulant graphs that admit a P1F. Herke \cite{SaraPhDThesis} also determined all connection sets of 3- and 4-regular circulant graphs on \(2n\) vertices that admit U1Fs for \(n \leq 14\). She also showed that any connected circulant graph isomorphic to \(Circ(2n,\left\lbrace 2,n \right\rbrace)\) admits a U1F if and only if \(n\equiv 3 \pmod {6}\), and provided a number of constructions of U1Fs of 3- and 4-regular graphs \cite{SaraPhDThesis}.

\begin{theorem}\textup{\cite{Herke2013P1FCirc}}
    Let \(n \geq 2 \) be an even integer and let \(a\) be an integer in the range \(1 \leq a < n\). Then \(Circ(2n,\left\lbrace a,n \right\rbrace)\) has a P1F if and only if one of the following cases holds true:
    \begin{itemize}
        \item \(n=2\) and \(a=1\);
        \item \(n=3\) and \(a=1\) or \(a=2\);
        \item \(n>3\), \(n\) and \( a\) are odd and \(\gcd(a,\frac{n}{2})=1\).
    \end{itemize}
\end{theorem}

For 4-regular circulant graphs, Herke gave constructions of P1Fs for a number of infinite families of bipartite 4-regular circulant graphs and ruled out the existence of P1Fs for certain families of non-bipartite 4-regular circulant graphs \cite{HerkeP1F4RegCirc}.

There have been a number of generalisations of U1Fs and P1Fs of graphs. A \emph{semiperfect} 1-factorisation of a graph is a 1-factorisation \(\mathcal{F}=\left\lbrace F_1,F_2, \dots, F_\alpha  \right\rbrace\) of an \(\alpha\)-regular graph such that \(F_1\cup F_i \) forms a Hamilton cycle for \(2 \leq i \leq \alpha\). Ehrenfeucht showed in 1973 that semiperfect 1-factorisations exist for all complete graphs of even order \cite{EhrenfeuchtSemiperfect1973}. A \emph{sequentially} uniform 1-factorisation of a graph is a 1-factorisation \(\mathcal{F}\) of an \(\alpha\)-regular graph whose 1-factors can be labelled \(F_1,F_2, \dots, F_\alpha\) such that \(F_i\cup F_{i+1} \) is isomorphic to some 2-regular graph \(H\) for all \(1 \leq i \leq \alpha\) (where \(\alpha+1 = 1\)). If \(H\) is a Hamilton cycle we call such a 1-factorisation sequentially perfect.

The goal of this paper is to introduce a new type of 1-factorisation that is a relaxation of a uniform 1-factorisation.
We say that a pair of 1-factors of a 1-factorisation of a graph has \emph{type} $\left[c_1,c_2,\dots,c_t\right]$ if their union is a collection of $t$ cycles of lengths $c_1,c_2,\dots,c_t$, and we use the notation \(\left[c_1^{n_1},c_2^{n_2}, \dots, c_t^{n_t}\right]\) for a type with \(n_i\) cycles of length \(c_i\) for each \(i\in \left\lbrace 1,2, \dots, t  \right\rbrace\).
Consider a 1-factorisation $\mathcal{F}$ of some graph and let $T_1,T_2,\dots ,T_m$ be the types of the pairs of 1-factors of $\mathcal{F}$.
Let $a_{i}$ be the number of pairs of 1-factors that are of type $T_i$.
If $a_{1}=a_{2}=\dots=a_{m}=b$ for some integer $b$, then we say that $\mathcal{F}$ is a \emph{\(m\)-balanced 1-factorisation} (\(m\)-B1F) with types \(\left(T_1,T_2,\dots ,T_m\right)\). If it is not important how many types appear, then we will just use B1F.

Clearly, \(m\)-B1Fs of \(r\)-regular graphs are only feasible when \(m\mid \binom{r}{2}\) and we note that a \(1\)-B1F is a U1F. We begin the investigation of \(m\)-B1Fs by considering a family of highly symmetric graphs with low degree, specifically, 3- and 4-regular connected circulant graphs.

\section{B1Fs of Connected 3-Regular Circulant Graphs}
\label{Subsec:B1Fs of 3-Regular Circulant Graphs}
As noted, \(m\)-B1Fs of \(r\)-regular graphs require \(m \mid \binom{r}{2}\). Thus, an \(m\)-B1F of a 3-regular graph can only exist if \(m\in \left\lbrace 1,3 \right\rbrace\). Herke \cite{Herke2013P1FCirc} has considered the problem of finding 1-B1Fs of connected 3-regular circulant graphs before, so we shall focus on the case where \(m=3\) and determine the values of \(n\) for which 3-B1Fs of connected 3-regular circulant graphs on \(2n\) vertices exist.
Herke and Maenhaut \cite{Herke2013P1FCirc} have shown that up to isomorphism there are only two connected 3-regular circulant graphs of order $2n$, $Circ(2n, \left\lbrace 1,n \right\rbrace)$ and $Circ(2n, \left\lbrace 2,n \right\rbrace)$. We completely answer the question of existence of 3-B1Fs of connected 3-regular circulant graphs with the following two theorems.

\begin{theorem} \label{Thm: Circ{n,{1,n/2}}}
    The circulant graph \(Circ(2n, \left\lbrace 1,n \right\rbrace)\) admits a 3-B1F if and only if $n=6$ or \(n \geq 8\).
\end{theorem}
\begin{proof}
    For \(Circ(2n, \left\lbrace 1,n \right\rbrace)\) to admit a 3-B1F, clearly there must be at least three possible types of pairs of 1-factors. The union of a pair of 1-factors will be a collection of even length cycles of length at least 4. If \(2n \leq 10\) it is impossible to have three distinct types of pairs of 1-factors of \(Circ(2n, \left\lbrace 1,n \right\rbrace)\). For \(n=7\), we enumerated all possible 1-factorisations of \(Circ(2n, \left\lbrace 1,n \right\rbrace)\) by computer and determined that none are 3-B1Fs. It remains to show that we can construct 3-B1Fs for all remaining even values of \(n\).

    We consider two separate cases: even \(n\) where \(n \geq 6 \), and odd \(n\) where \(n \geq 9 \).

    First, if \(n\) is even and \(n \geq 12 \), consider the 1-factorisation \(\mathcal{F}_\alpha = \{F_1,F_2,F_3\}\), of \(Circ(2n, \left\lbrace 1,n \right\rbrace)\) defined as follows:
    \begin{align*}
        F_1 & = \left\lbrace \left\lbrace x,x + 1\right\rbrace \, :\, 0 \leq x \leq 2n - 2,\, \text{$x$ even} \right\rbrace                                                                 \\
        F_2 & = \left\lbrace \left\lbrace x,x + 1\right\rbrace \, :\, 1 \leq x \leq 2n - 3,\,  x \neq n - 1,\, \text{$x$ odd} \right\rbrace \cup \left\lbrace \left\lbrace x, x+ n \right\rbrace\, :\, x \in \left\lbrace 0,n- 1 \right \rbrace \right\rbrace \\
        F_3 & = \left\lbrace \left\lbrace x,x + 1 \right\rbrace\, :\, x \in \left\lbrace n - 1, 2n- 1 \right \rbrace \right\rbrace       \cup \left\lbrace \left\lbrace x, x+ n \right\rbrace \, :\, 1 \leq x \leq n - 2\right\rbrace.
    \end{align*}

    To show that \(\mathcal{F}_\alpha\) is a 3-B1F we will show that $F_1\cup F_2$, $F_1\cup F_3$, and $F_2\cup F_3$ are each of different type.

    \begin{figure}[h]
        \centering
        \begin{subfigure}[ht!]{\textwidth}
            \centering
            \scalebox{0.95}{
                \begin{tikzpicture}[dot/.style={circle, draw=none, inner sep=0pt, minimum size=1pt},ddot/.style={circle,fill=black,inner sep=0pt,minimum size=8pt},
                        lbl/.style={}]
                    \def\n{28} 
                    \def\fz{3} 
                    \def\lfz{2} 
                    \def\bz{5} 
                    \def\lbz{3} 
                    \def\fh{\fz} 
                    \def\lfh{\lfz} 
                    \def\bh{\bz} 
                    \def\lbh{\lbz} 

                    \def\radius{5cm} 
                    \def\radiustwo{5.5cm} 
                    \foreach \s in {0,...,\numexpr\lfz\relax}
                        {
                            \node[dot] (\s) at ({360/\n * (-\s)+90}:\radius) {};
                            \node[draw=none] (label\s) at ({360/\n * (-\s)+90}:\radiustwo) {\small\(\s\)};
                        }
                    \foreach \s in {1,...,\numexpr\lbz\relax}
                        {
                            \node[dot] (nminus\s) at ({360/\n * (\s)+90}:\radius) {};
                            \node[draw=none] (labelnminus\s) at ({360/\n * (\s)+90}:\radiustwo) {\small\(2n-\s\)};
                        }
                    \foreach \s in {\lfz,...,\numexpr\fz+1\relax}
                        {
                            \node[dot] (\s) at ({360/\n * (-\s)+90}:\radius) {};
                        }
                    \foreach \s in {\lbz,...,\numexpr\bz+1\relax}
                        {
                            \node[dot] (nminus\s) at ({360/\n * (\s)+90}:\radius) {};
                        }
                    \foreach \s in {0,...,\numexpr\fz\relax}
                        {
                            \node[ddot] (disp\s) at ({360/\n * (-\s)+90}:\radius) {};
                        }
                    \foreach \s in {1,...,\numexpr\bz\relax}
                        {
                            \node[ddot] (disp\n-\s) at ({360/\n * (\s)+90}:\radius) {};
                        }
                    \node[dot] (halfn) at ({360/\n * (0)-90}:\radius) {};
                    \node[draw=none] (labelhalfn) at ({360/\n * (0)-90}:\radiustwo) {\(n\)};
                    \foreach \s in {1,...,\numexpr\lfh\relax}
                        {
                            \node[dot] (halfnplus\s) at ({360/\n * (-\s)-90}:\radius) {};
                            \node[draw=none] (labelhalfnplus\s) at ({360/\n * (-\s)-90}:\radiustwo) {\(n+\s\)};
                        }
                    \foreach \s in {1,...,\numexpr\lbh\relax}
                        {
                            \node[dot] (halfnminus\s) at ({360/\n * (\s)-90}:\radius) {};
                            \node[draw=none] (labelhalfnminus\s) at ({360/\n * (\s)-90}:\radiustwo) {\(n-\s\)};
                        }
                    \foreach \s in {\lfh,...,\numexpr\fh+1\relax}
                        {
                            \node[dot] (halfnplus\s) at ({360/\n * (-\s)-90}:\radius) {};
                        }
                    \foreach \s in {\lbh,...,\numexpr\bh+1\relax}
                        {
                            \node[dot] (halfnminus\s) at ({360/\n * (\s)-90}:\radius) {};
                        }
                    \foreach \s in {0,...,\numexpr\fh\relax}
                        {
                            \node[ddot] (disphalfnplus\s) at ({360/\n * (-\s)-90}:\radius) {};
                        }
                    \foreach \s in {1,...,\numexpr\bh\relax}
                        {
                            \node[ddot] (disphalfnminus\s) at ({360/\n * (\s)-90}:\radius) {};
                        }

                    \begin{scope}[on background layer]
                        \foreach \s/\y in  {0/1,2/3,halfn/halfnplus1,halfnplus2/halfnplus3,halfnminus2/halfnminus1,halfnminus4/halfnminus3,halfnminus6/halfnminus5,nminus2/nminus1,nminus4/nminus3,nminus6/nminus5}
                            {
                                \draw [circF1]  (\s) -- (\y);

                            }
                        \foreach \s/\y in {1/2,3/4,halfnminus5/halfnminus4,halfnminus3/halfnminus2,halfnminus1/nminus1,0/halfn,halfnplus1/halfnplus2,halfnplus3/halfnplus4,nminus5/nminus4,nminus3/nminus2}
                            {
                                \draw [circF2]  (\s) -- (\y);
                            }

                        \foreach \s/\y in {0/nminus1,halfnminus1/halfn,1/halfnplus1,2/halfnplus2,3/halfnplus3,nminus2/halfnminus2,nminus3/halfnminus3,nminus4/halfnminus4,nminus5/halfnminus5}
                            {
                                \draw [circF3]  (\s) -- (\y);

                            }
                    \end{scope}
                    \begin{scope}[node distance=1cm, every node/.style={font=\sffamily}, align=left]
                        \matrix [above right=of current bounding box.north east, yshift=-0.5cm,xshift=-0.75cm,anchor=north west, nodes={inner sep=0pt}, row sep=0.15cm] {
                        \node [label=right:{\(F_1\)}] (legend1) {};
                        \draw [circF1] ([xshift=-1.1cm]legend1.west) -- ([xshift=-0.1cm]legend1.west);
                        \\
                        \node [label=right:{\(F_2\)}] (legend2) {};
                        \draw [circF2] ([xshift=-1.1cm]legend2.west) -- ([xshift=-0.1cm]legend2.west);
                        \\
                        \node [label=right:{\(F_3\)}] (legend3) {};
                        \draw [circF3] ([xshift=-1.1cm]legend3.west) -- ([xshift=-0.1cm]legend3.west);
                        \\
                        };
                    \end{scope}
                \end{tikzpicture}
            }

        \end{subfigure}
        \caption{A 3-B1F of \(Circ(2n,\left\lbrace 1,n \right\rbrace)\) for even \(n\).}\label{Fig: 0 mod 4 General Construction 3-B1F Circ(n,{1,n/2})}
    \end{figure}
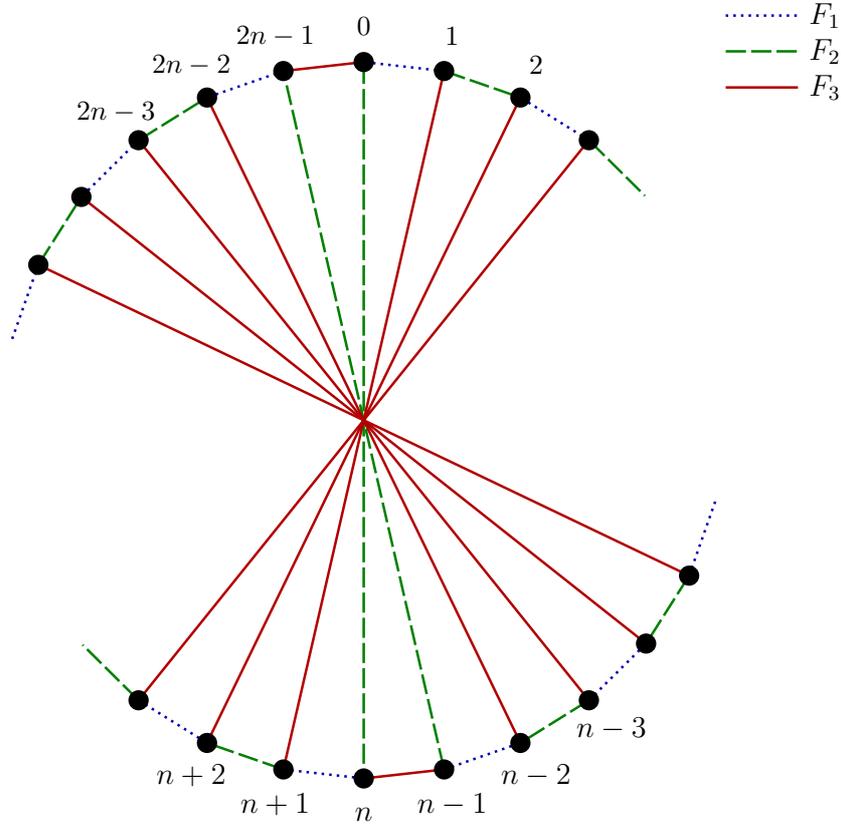

    The union $F_1\cup F_2$ is of type \([2n]\) with the Hamilton cycle $$\left(0,1,2,\dots,n-1,2n-1,2n-2,\dots, n\right).$$ The union \(F_2 \cup F_3\) is of type \(\left[4^{\tfrac{2n}{4}}\right]\) with cycles \[\left(x,x+1,x+1+n,x+n\right)\text{ for }x\in \left\lbrace 1, 3, \dots, n-1  \right\rbrace.\]
    The union \({F_1\cup F_3}\) is of type \(\left[8,4^{\tfrac{2n-8}{4}}\right]\) with cycles
    \begin{align*}
        \left(0,1,n+1, n, n-1, n-2,\allowbreak 2n-2, 2n-1\right)\text{ and }\\ \left(x,x+1,x+1+n,x+n\right) \text{ for } x \in \left\lbrace 2, 4, \dots, n-4 \right\rbrace.
    \end{align*}
    Thus, $\mathcal{F}_\alpha$ is a 3-B1F with types \(\left([2n],\left[4^{\tfrac{2n}{4}}\right],\left[8,4^{\tfrac{2n-8}{4}}\right]\right)\).

    If \(n\) is odd and \(n \geq 9 \), consider the 1-factorisation \(\mathcal{F}_\beta = \{F_a,F_b, F_c\}\) defined as follows:
    \begin{align*}
        F_a & = \left\lbrace
        \left\lbrace
        x,x + 1
        \right\rbrace
        \, :\, 0 \leq x \leq n - 3,\,  \text{$x$ even}
        \right\rbrace
        \cup
        \left\lbrace
        \left\lbrace
        x,x + 1
        \right\rbrace
        \, :\, n \leq x \leq 2n - 3,\,  \text{$x$ odd}
        \right\rbrace
        \\
            & \qquad \cup
        \left\lbrace
        \left\lbrace n - 1,2n-1 \right\rbrace
        \right\rbrace
        \\
        F_b & = \left\lbrace
        \left\lbrace
        x,x + 1
        \right\rbrace
        \, :\, 1 \leq  x \leq  n-2 , \, x \neq n - 4,\,   \text{$x$ odd}
        \right\rbrace
        \\ & \qquad \cup
        \left\lbrace
        \left\lbrace
        x,x + 1
        \right\rbrace
        \, :\, n+1 \leq  x \leq 2n-2 , \, x \neq 2n - 4, \, \text{$x$ even}
        \right\rbrace
        \\ & \qquad \cup
        \left\lbrace
        \left\lbrace x,x+n \right\rbrace\,:\, x \in
        \left\lbrace
        0,
        n - 4,
        n - 3
        \right\rbrace
        \right\rbrace
        \\
        F_c & =
        \left\lbrace
        \left\lbrace x,x+1 \right\rbrace\,:\, x \in
        \left\lbrace
        n-4,n-1, 2n-4,2n-1
        \right\rbrace
        \right\rbrace                  \\
            & \qquad \cup \left\lbrace
        \left\lbrace
        x,x+n
        \right\rbrace
        \, :\, 1 \leq x \leq n - 2,\, x\not\in \{n - 4,n - 3\}
        \right\rbrace
    \end{align*}

    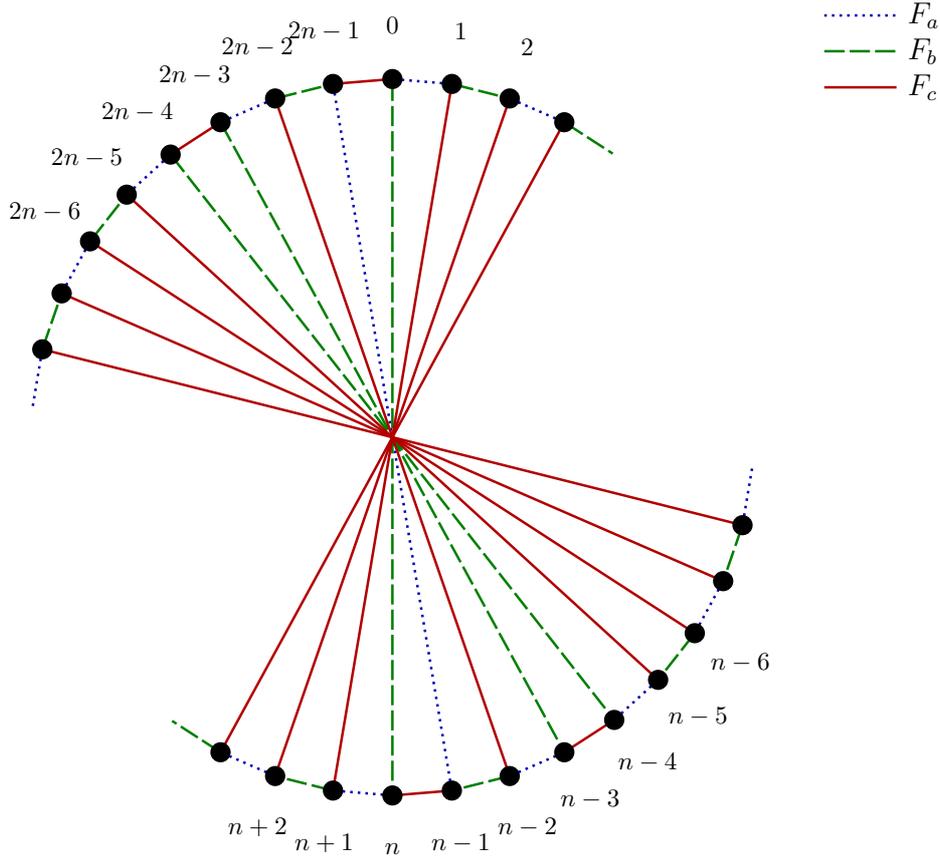
\begin{figure}[h]
        \centering
        \begin{subfigure}[b]{\textwidth}
            \centering
            \scalebox{0.95}{
                \begin{tikzpicture}[dot/.style={circle,draw = none, inner sep=0pt, minimum size=1pt},ddot/.style={circle,fill=black,inner sep=0pt,minimum size=8pt},
                        lbl/.style={}]
                    \def\n{38} 
                    \def\fz{3} 
                    \def\lfz{2} 
                    \def\bz{8} 
                    \def\lbz{6} 
                    \def\fh{\fz} 
                    \def\lfh{\lfz} 
                    \def\bh{\bz} 
                    \def\lbh{\lbz} 

                    \def\radius{5cm} 
                    \def\radiustwo{5.75cm} 
                    \foreach \s in {0,...,\numexpr\lfz\relax}
                        {
                            \node[dot] (\s) at ({360/\n * (-\s)+90}:\radius) {};
                            \node[draw=none] (label\s) at ({360/\n * (-\s)+90}:\radiustwo) {\footnotesize\(\s\)};
                        }
                    \foreach \s in {1,...,\numexpr\lbz\relax}
                        {
                            \node[dot] (nminus\s) at ({360/\n * (\s)+90}:\radius) {};
                            \node[draw=none] (labelnminus\s) at ({360/\n * (\s)+90}:\radiustwo) {\footnotesize \(2n-\s\)};
                        }
                    \foreach \s in {\lfz,...,\numexpr\fz+1\relax}
                        {
                            \node[dot] (\s) at ({360/\n * (-\s)+90}:\radius) {};
                        }
                    \foreach \s in {\lbz,...,\numexpr\bz+1\relax}
                        {
                            \node[dot] (nminus\s) at ({360/\n * (\s)+90}:\radius) {};
                        }
                    \foreach \s in {0,...,\numexpr\fz\relax}
                        {
                            \node[ddot] (disp\s) at ({360/\n * (-\s)+90}:\radius) {};
                        }
                    \foreach \s in {1,...,\numexpr\bz\relax}
                        {
                            \node[ddot] (disp\n-\s) at ({360/\n * (\s)+90}:\radius) {};
                        }
                    \node[dot] (halfn) at ({360/\n * (0)-90}:\radius) {};
                    \node[draw=none] (labelhalfn) at ({360/\n * (0)-90}:\radiustwo) {\footnotesize\(n\)};
                    \foreach \s in {1,...,\numexpr\lfh\relax}
                        {
                            \node[dot] (halfnplus\s) at ({360/\n * (-\s)-90}:\radius) {};
                            \node[draw=none] (labelhalfnplus\s) at ({360/\n * (-\s)-90}:\radiustwo) {\footnotesize\(n+\s\)};
                        }
                    \foreach \s in {1,...,\numexpr\lbh\relax}
                        {
                            \node[dot] (halfnminus\s) at ({360/\n * (\s)-90}:\radius) {};
                            \node[draw=none] (labelhalfnminus\s) at ({360/\n * (\s)-90}:\radiustwo) {\footnotesize\(n-\s\)};
                        }
                    \foreach \s in {\lfh,...,\numexpr\fh+1\relax}
                        {
                            \node[dot] (halfnplus\s) at ({360/\n * (-\s)-90}:\radius) {};
                        }
                    \foreach \s in {\lbh,...,\numexpr\bh+1\relax}
                        {
                            \node[dot] (halfnminus\s) at ({360/\n * (\s)-90}:\radius) {};
                        }
                    \foreach \s in {0,...,\numexpr\fh\relax}
                        {
                            \node[ddot] (disphalfnplus\s) at ({360/\n * (-\s)-90}:\radius) {};
                        }
                    \foreach \s in {1,...,\numexpr\bh\relax}
                        {
                            \node[ddot] (disphalfnminus\s) at ({360/\n * (\s)-90}:\radius) {};
                        }

                    \begin{scope}[on background layer]
                        \foreach \s/\y in  {nminus1/halfnminus1,0/1,2/3,halfn/halfnplus1,halfnplus2/halfnplus3,halfnminus9/halfnminus8,halfnminus7/halfnminus6,halfnminus5/halfnminus4,halfnminus3/halfnminus2,nminus9/nminus8,nminus7/nminus6,nminus5/nminus4,nminus3/nminus2}
                            {
                                \draw [circF1]  (\s) -- (\y);

                            }
                        \foreach \s/\y in {1/2,3/4,0/halfn,nminus3/halfnminus3,nminus4/halfnminus4,halfnminus8/halfnminus7,halfnminus6/halfnminus5,halfnminus2/halfnminus1,halfnplus1/halfnplus2,halfnplus3/halfnplus4,nminus8/nminus7,nminus6/nminus5,nminus2/nminus1}
                            {
                                \draw [circF2]  (\s) -- (\y);
                            }

                        \foreach \s/\y in {0/nminus1,halfnminus1/halfn,1/halfnplus1,2/halfnplus2,3/halfnplus3,nminus2/halfnminus2,nminus5/halfnminus5,nminus6/halfnminus6,nminus7/halfnminus7,nminus8/halfnminus8,nminus4/nminus3,halfnminus4/halfnminus3}
                            {
                                \draw [circF3]  (\s) -- (\y);

                            }
                    \end{scope}
                    \begin{scope}[node distance=1cm, every node/.style={font=\sffamily}, align=left]
                        \matrix [above right=of current bounding box.north east, yshift=-0.5cm,xshift=-0.25cm,anchor=north west, nodes={inner sep=0pt}, row sep=0.15cm] {
                        \node [label=right:{\(F_a\)}] (legend1) {};
                        \draw [circF1] ([xshift=-1.1cm]legend1.west) -- ([xshift=-0.1cm]legend1.west);
                        \\
                        \node [label=right:{\(F_b\)}] (legend2) {};
                        \draw [circF2] ([xshift=-1.1cm]legend2.west) -- ([xshift=-0.1cm]legend2.west);
                        \\
                        \node [label=right:{\(F_c\)}] (legend3) {};
                        \draw [circF3] ([xshift=-1.1cm]legend3.west) -- ([xshift=-0.1cm]legend3.west);
                        \\
                        };
                    \end{scope}
                \end{tikzpicture}
            }

        \end{subfigure}
        \caption{A 3-B1F of \(Circ(2n,\left\lbrace 1,n \right\rbrace)\) for odd \(n\).}\label{Fig: 2 mod 4 General Construction 3-B1F Circ(n,{1,n/2})}
    \end{figure}

    The union $F_a\cup F_b$ is of type \(\left[2n-6,6\right]\)
    with cycles \begin{align*}
        \left(0,1,2,3,4,\dots, n-4,\allowbreak 2n-4, 2n-5, 2n-6, \dots , n\right) \text{ and } \\ \left(n-3,n-2,n-1,2n-1,2n-2,2n-3\right).
    \end{align*}
    The union $F_b\cup F_c$ is of type \(\left[6,4^{\tfrac{2n-6}{4} }\right]\)
    with cycles \begin{align*}\left(0,2n-1,\allowbreak 2n-2,n-2,n-1, n\right)\text{ and
        }\\ \left(x,x+1,\allowbreak x+1+n, x+n\right)\text{ for } x\in \left\lbrace 1, 3, \dots, n-4 \right\rbrace.\end{align*}
    The union $F_a\cup F_c$ is of type \(\left[8,6,4^{\tfrac{2n-14}{4} }\right]\) with cycles
    \begin{gather*}
        \left(2n-2,2n-3,2n-4,2n-5,n-5,n-4,n-3,n-2\right), \left(0,1,n+1,n,n-1,2n-1\right), \\\text{ and }
        \left(x,x+1,x+1+n,x+n\right)\text{ for } x \in \left\lbrace 2, 4, \dots, n-7  \right\rbrace.
    \end{gather*}
    Thus, $\mathcal{F}_\beta$ is a 3-B1F with types \(\left(\left[2n-6,6\right],\left[6,4^{\tfrac{2n-6}{4} }\right],\left[8,6,4^{\tfrac{2n-14}{4} }\right]\right)\).
    
    Therefore the connected 3-regular circulant graph \(Circ(2n,\left\lbrace 1,n  \right\rbrace)\) admits a 3-B1F if and only if \(n=6\) or \(n \geq 8\).
\end{proof}

\begin{theorem}\label{Thm: Circ{n,{2,n/2}}}
    The connected circulant graph \(Circ(2n, \left\lbrace 2,n \right\rbrace)\) admits a 3-B1F if and only if \(n \geq 9 \).
\end{theorem}
\begin{proof}
    If \(Circ(2n,\left\lbrace 2,n \right\rbrace)\) is connected, then \(\gcd(2n,2,n)=1\). Thus, we now only consider odd \(n\). For \(Circ(2n, \left\lbrace 2,n \right\rbrace)\) to admit a 3-B1F, clearly there must be at least three possible types of pairs of 1-factors. The union of a pair of 1-factors will be a collection of even length cycles of length at least 4. It is clear that if \(2n \leq 10\), it is impossible to have three distinct types of pairs of 1-factors of \(Circ(2n, \left\lbrace 2,n \right\rbrace)\). For \(n=7\), we enumerated all non-isomorphic 1-factorisations of \(Circ(2,\{2,n\})\) by computer and determined that none are 3-B1Fs. It remains to show that we can construct 3-B1Fs for all remaining even \(n \geq 9 \).

    We consider this in two separate cases: when \(n \equiv 1 \pmod 4\), and when \(n \equiv 3 \pmod 4\).

    First, if \(n\equiv 1 \pmod 4\), consider the 1-factorisation \(\mathcal{F}_\alpha = \{F_1,F_2,F_3\}\) of \({Circ}(2n,\{2,n\})\) defined as follows:
    \begin{align*}
        F_1 & = \left\lbrace
        \left\lbrace
        x,x+2
        \right\rbrace\, :\, 0 \leq x \leq n-4,\, x\equiv 0,1 \Mod 4
        \right\rbrace                                                                                                                                    \\
            & \qquad \cup \left\lbrace \left\lbrace x,x+2\right\rbrace\, :\,n \leq x \leq 2n-4,\, x\equiv 1,2 \Mod 4 \right\rbrace             \\
            & \qquad \cup \left\lbrace
            \left\lbrace n-1,2n-1 \right\rbrace
        \right\rbrace                                                                                                                                    \\
        F_2 & = \left\lbrace \left\lbrace x,x+2 \right\rbrace\,:\, x \in \left\lbrace n-7,n-3, 2n-7,2n-3 \right\rbrace \right\rbrace \\
            & \qquad \cup
        \left\lbrace \left\lbrace x,x+n  \right\rbrace\,:\, 0 \leq x \leq n -8
        \right\rbrace                                                                                                                                    \\
            & \qquad \cup
        \left\lbrace \left\lbrace x,x+n  \right\rbrace\,:\, x \in \left\lbrace n-6,n-4,n-2 \right\rbrace
        \right\rbrace                                                                                                                                    \\
        F_3 & = \left\lbrace
        \left\lbrace
        x,x+2
        \right\rbrace\, :\, 0 \leq x \leq n-10,\, x\equiv 2,3 \Mod 4
        \right\rbrace                                                                                                                                    \\
            & \qquad \cup \left\lbrace \left\lbrace x,x+2\right\rbrace\, :\,n \leq x \leq 2n-10,\, x\equiv 0,3 \Mod 4 \right\rbrace            \\
            & \qquad \cup \left\lbrace
        \left\lbrace x,x+2 \right\rbrace\,:\,x \in
        \left\lbrace
        n - 6,
        n - 2,
        n - 1,
        2n- 6,
        2n- 2,
        2n- 1
        \right\rbrace
        \right\rbrace                                                                                                                                    \\
            & \qquad \cup \left\lbrace
        \left\lbrace x,x+n \right\rbrace\,:\,x \in
        \left\lbrace
        n - 7,
        n - 5,
        n - 3
        \right\rbrace
        \right\rbrace .
    \end{align*}

    \begin{figure}[hb!]
        \centering
        \begin{subfigure}[b]{\textwidth}
            \centering
            \scalebox{0.95}{
                \begin{tikzpicture}[dot/.style={draw=none,circle, inner sep=0pt, minimum size=1pt},ddot/.style={circle,fill=black,inner sep=0pt,minimum size=8pt},
                        lbl/.style={}]
                    \def\n{46} 
                    \def\fz{5} 
                    \def\lfz{3} 
                    \def\bz{9} 
                    \def\lbz{8} 
                    \def\fh{\fz} 
                    \def\lfh{\lfz} 
                    \def\bh{\bz} 
                    \def\lbh{\lbz} 

                    \def\radius{6cm} 
                    \def\radiustwo{6.7cm} 
                    \foreach \s in {0,...,\numexpr\lfz\relax}
                        {
                            \node[dot] (\s) at ({360/\n * (-\s)+90}:\radius) {};
                            \node[draw=none] (label\s) at ({360/\n * (-\s)+90}:\radiustwo) {\small\(\s\)};
                        }
                    \foreach \s in {1,...,\numexpr\lbz\relax}
                        {
                            \node[dot] (nminus\s) at ({360/\n * (\s)+90}:\radius) {};
                            \node[draw=none] (labelnminus\s) at ({360/\n * (\s)+90}:\radiustwo) {\small\(2n-\s\)};
                        }
                    \foreach \s in {\lfz,...,\numexpr\fz+2\relax}
                        {
                            \node[dot] (\s) at ({360/\n * (-\s)+90}:\radius) {};
                        }
                    \foreach \s in {\lbz,...,\numexpr\bz+2\relax}
                        {
                            \node[dot] (nminus\s) at ({360/\n * (\s)+90}:\radius) {};
                        }
                    \foreach \s in {0,...,\numexpr\fz\relax}
                        {
                            \node[ddot] (disp\s) at ({360/\n * (-\s)+90}:\radius) {};
                        }
                    \foreach \s in {1,...,\numexpr\bz\relax}
                        {
                            \node[ddot] (disp\n-\s) at ({360/\n * (\s)+90}:\radius) {};
                        }
                    \node[dot] (halfn) at ({360/\n * (0)-90}:\radius) {};
                    \node[draw=none] (labelhalfn) at ({360/\n * (0)-90}:\radiustwo) {\small\(n\)};
                    \foreach \s in {1,...,\numexpr\lfh\relax}
                        {
                            \node[dot] (halfnplus\s) at ({360/\n * (-\s)-90}:\radius) {};
                            \node[draw=none] (labelhalfnplus\s) at ({360/\n * (-\s)-90}:\radiustwo) {\small\(n+\s\)};
                        }
                    \foreach \s in {1,...,\numexpr\lbh\relax}
                        {
                            \node[dot] (halfnminus\s) at ({360/\n * (\s)-90}:\radius) {};
                            \node[draw=none] (labelhalfnminus\s) at ({360/\n * (\s)-90}:\radiustwo) {\small\(n-\s\)};
                        }
                    \foreach \s in {\lfh,...,\numexpr\fh+2\relax}
                        {
                            \node[dot] (halfnplus\s) at ({360/\n * (-\s)-90}:\radius) {};
                        }
                    \foreach \s in {\lbh,...,\numexpr\bh+2\relax}
                        {
                            \node[dot] (halfnminus\s) at ({360/\n * (\s)-90}:\radius) {};
                        }
                    \foreach \s in {0,...,\numexpr\fh\relax}
                        {
                            \node[ddot] (disphalfnplus\s) at ({360/\n * (-\s)-90}:\radius) {};
                        }
                    \foreach \s in {1,...,\numexpr\bh\relax}
                        {
                            \node[ddot] (disphalfnminus\s) at ({360/\n * (\s)-90}:\radius) {};
                        }

                    \begin{scope}[on background layer]
                        \foreach \s/\y in  {nminus1/halfnminus1}
                            {
                                \draw [circF1]  (\s) -- (\y);

                            }
                        \foreach \s/\y in  {0/2,1/3,4/6,5/7,halfnminus9/halfnminus7,halfnminus8/halfnminus6,halfnminus5/halfnminus3,halfnminus4/halfnminus2,halfn/halfnplus2,halfnplus1/halfnplus3,halfnplus4/halfnplus6,halfnplus5/halfnplus7,nminus9/nminus7,nminus8/nminus6,nminus5/nminus3,nminus4/nminus2}
                            {
                                \draw [circF1]  (\s) to[bend left=-45] (\y);

                            }
                        \foreach \s/\y in {0/halfn,1/halfnplus1,2/halfnplus2,3/halfnplus3,4/halfnplus4,5/halfnplus5, nminus2/halfnminus2,nminus4/halfnminus4,nminus6/halfnminus6,nminus8/halfnminus8,nminus9/halfnminus9}
                            {
                                \draw [circF2]  (\s) -- (\y);
                            }
                        \foreach \s/\y in  {halfnminus7/halfnminus5,halfnminus3/halfnminus1,nminus7/nminus5,nminus3/nminus1}
                            {
                                \draw [circF2]  (\s) to[bend left=-45] (\y);

                            }
                        \foreach \s/\y in {nminus7/halfnminus7,nminus5/halfnminus5,nminus3/halfnminus3}
                            {
                                \draw [circF3]  (\s) -- (\y);

                            }
                        \foreach \s/\y in  {nminus11/nminus9,nminus10/nminus8,nminus6/nminus4,nminus2/0,nminus1/1,2/4,3/5,halfnminus11/halfnminus9,halfnminus10/halfnminus8,halfnminus6/halfnminus4,halfnminus2/halfn,halfnminus1/halfnplus1,halfnplus2/halfnplus4,halfnplus3/halfnplus5}
                            {
                                \draw [circF3]  (\s) to[bend left=-45] (\y);

                            }
                    \end{scope}
                    \begin{scope}[node distance=1cm, every node/.style={font=\sffamily}, align=left]
                        \matrix [above right=of current bounding box.north east, yshift=-0.5cm,xshift=-0.25cm,anchor=north west, nodes={inner sep=0pt}, row sep=0.15cm] {
                        \node [label=right:{\(F_1\)}] (legend1) {};
                        \draw [circF1] ([xshift=-1.1cm]legend1.west) -- ([xshift=-0.1cm]legend1.west);
                        \\
                        \node [label=right:{\(F_2\)}] (legend2) {};
                        \draw [circF2] ([xshift=-1.1cm]legend2.west) -- ([xshift=-0.1cm]legend2.west);
                        \\
                        \node [label=right:{\(F_3\)}] (legend3) {};
                        \draw [circF3] ([xshift=-1.1cm]legend3.west) -- ([xshift=-0.1cm]legend3.west);
                        \\
                        };
                    \end{scope}
                \end{tikzpicture}
            }
        \end{subfigure}
        \caption{Construction of 3-B1F of \(Circ(2n,\left\lbrace 2,n \right\rbrace)\) for \(n\equiv 1 \pmod 4\).}
        \label{Fig: 2 mod 8 General Construction 3-B1F Circ(n,{2,n/2})}
    \end{figure}
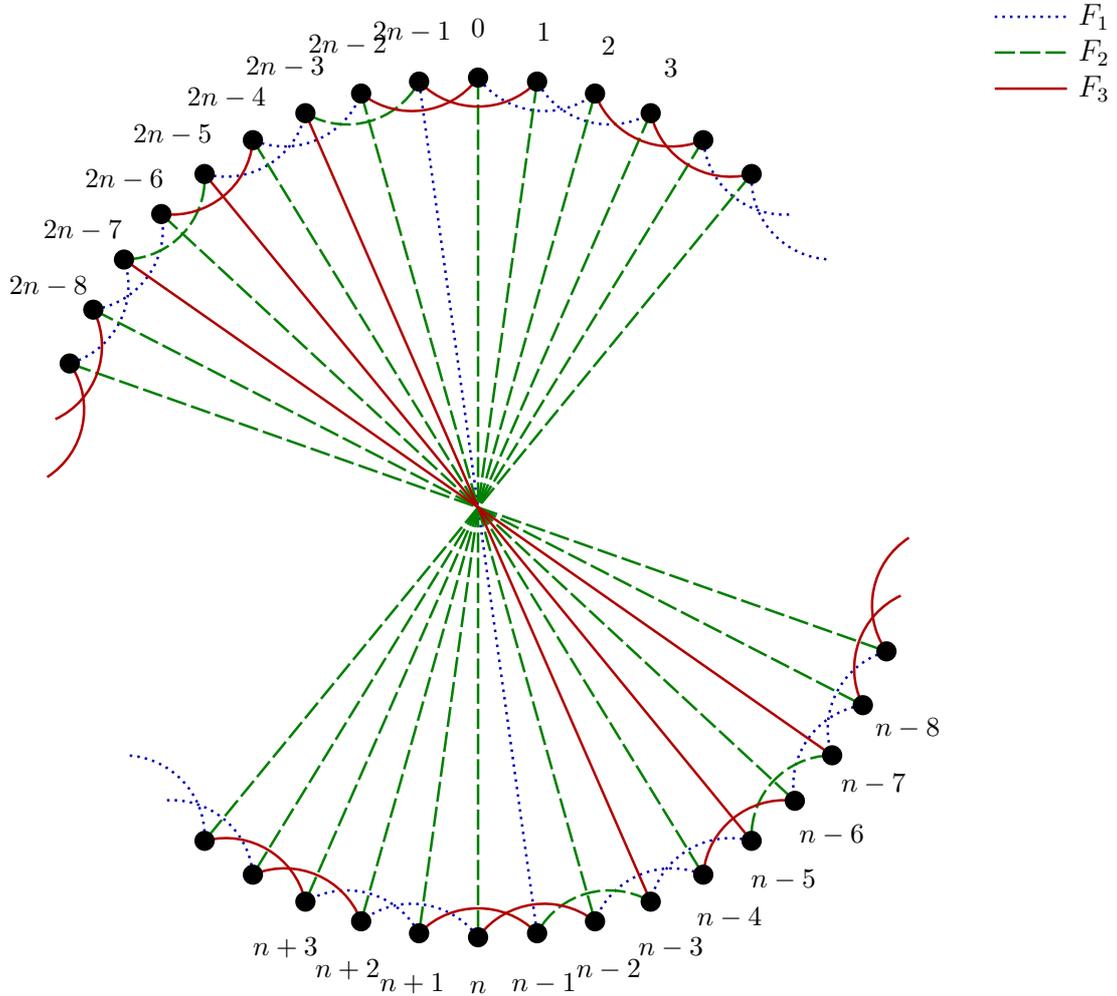
    To show that \(\mathcal{F}_\alpha\) is a 3-B1F it suffices to show that $F_1\cup F_2$, $F_1\cup F_3$, and $F_2\cup F_3$ are all non-isomorphic graphs.

    The union $F_1 \cup F_2$ is of type \(\left[10,4^{\tfrac{2n-10}{4} }\right]\) with cycles
    \begin{gather*}
        \left(2n-1, 2n-3,2n-5,2n-7, 2n-9, n-9, n-7, n-5, n-3, n-1\right)\text{ and }\\
        \left( x,x+2,x+2+n,x+n\right)\text{ for } x \in \left\lbrace 0, 4, \dots, n-13 \right\rbrace \cup \left\lbrace 1, 5, \dots, n-4  \right\rbrace.
    \end{gather*}
    The union $F_1\cup F_3$ is of type \(\left[2n-4,4\right]\) with cycles
    \begin{gather*}
        \left(0,2,4,\dots, n-5, n-7,2n-7,2n-9,2n-11,\dots, 1,2n-1, n-1, n+1, n+3,\dots, 2n-2 \right)\\
        \text{ and }       \left(2n-5,2n-3,n-3,n-5\right).
    \end{gather*}
    The union $F_2\cup F_3$ is of type \(\left[6,4^{\tfrac{2n-6}{4} }\right]\) with cycles
    \begin{gather*}
        \left(1, 2n-1, 2n-3, n-3, n-1,n+1 \right)\text{ and }\\
        \left(x,x+2,n+x+2,n+x\right)\text{ for } x \in \left\lbrace 2,6, \dots, n-7  \right\rbrace \cup \left\lbrace 3,7, \dots, n-2  \right\rbrace.
    \end{gather*}

    Thus, $\mathcal{F}_\alpha$ is a 3-B1F with types \(\left(\left[10,4^{\tfrac{2n-10}{4} }\right],\left[2n-4,4\right],\left[6,4^{\tfrac{2n-6}{4} }\right]\right)\).

    If \(n\equiv 3 \pmod 4\), consider the 1-factorisation \(\mathcal{F}_\beta = \{F_a,F_b,F_c\}\) of \({Circ}(2n,\{2,n\})\) defined as follows:
    \begin{align*}
        F_a & = \left\lbrace
        \left\lbrace
        x,x+2
        \right\rbrace\, :\, 0 \leq x \leq n-6,\, x\equiv 0,1 \Mod 4
        \right\rbrace                                                                                                                                    \\
            & \qquad \cup \left\lbrace \left\lbrace x,x+2\right\rbrace\, :\,n \leq x \leq 2n-6,\, x\equiv 0,3 \Mod 4 \right\rbrace             \\
            & \qquad \cup \left\lbrace
        \left\lbrace x,x+2 \right\rbrace\,:\,x \in
        \left\lbrace
        n - 3,
        2n-3
        \right\rbrace
        \right\rbrace                                                                                                                                    \\
            & \qquad \cup \left\lbrace
            \left\lbrace n-2, 2n-2 \right\rbrace
        \right\rbrace                                                                                                                                    \\
        F_b & = \left\lbrace \left\lbrace x,x+2 \right\rbrace\,:\, x \in \left\lbrace n-8,n-4, 2n-8,2n-4 \right\rbrace \right\rbrace \\
            & \qquad \cup
        \left\lbrace \left\lbrace x,x+n  \right\rbrace\,:\, 0 \leq x \leq n -9
        \right\rbrace                                                                                                                                    \\
            & \qquad \cup
        \left\lbrace \left\lbrace x,x+n  \right\rbrace\,:\, x \in \left\lbrace n-7,n-5,n-3 ,n-1 \right\rbrace
        \right\rbrace                                                                                                                                    \\
        F_c & = \left\lbrace
        \left\lbrace
        x,x+2
        \right\rbrace\, :\, 0 \leq x \leq n-12,\, x\equiv 2,3 \Mod 4
        \right\rbrace                                                                                                                                    \\
            & \qquad \cup \left\lbrace \left\lbrace x,x+2\right\rbrace\, :\,n \leq x \leq 2n-12,\, x\equiv 1,2 \Mod 4 \right\rbrace            \\
            & \qquad \cup \left\lbrace
        \left\lbrace x,x+2 \right\rbrace\,:\,x \in
        \left\lbrace
        n - 9,
        n - 5,
        n - 2,
        n - 1,
        2n- 9,
        2n- 5,
        2n- 2,
        2n- 1
        \right\rbrace
        \right\rbrace                                                                                                                                    \\
            & \qquad \cup \left\lbrace
        \left\lbrace x,x+n \right\rbrace\,:\,x \in
        \left\lbrace
        n - 8,
        n - 6,
        n - 4
        \right\rbrace
        \right\rbrace .
    \end{align*}

    \begin{figure}[ht!]
        \centering
        \begin{subfigure}[b]{\textwidth}
            \centering
            \scalebox{0.95}{
                \begin{tikzpicture}[dot/.style={draw=none,circle, inner sep=0pt, minimum size=1pt},ddot/.style={circle,fill=black,inner sep=0pt,minimum size=8pt},
                        lbl/.style={}]
                    \def\n{46} 
                    \def\fz{5} 
                    \def\lfz{3} 
                    \def\bz{11} 
                    \def\lbz{7} 
                    \def\fh{\fz} 
                    \def\lfh{\lfz} 
                    \def\bh{\bz} 
                    \def\lbh{\lbz} 

                    \def\radius{6cm} 
                    \def\radiustwo{6.7cm} 
                    \foreach \s in {0,...,\numexpr\lfz\relax}
                        {
                            \node[dot] (\s) at ({360/\n * (-\s)+90}:\radius) {};
                            \node[draw=none] (label\s) at ({360/\n * (-\s)+90}:\radiustwo) {\small\(\s\)};
                        }
                    \foreach \s in {1,...,\numexpr\lbz\relax}
                        {
                            \node[dot] (nminus\s) at ({360/\n * (\s)+90}:\radius) {};
                            \node[draw=none] (labelnminus\s) at ({360/\n * (\s)+90}:\radiustwo) {\small\(2n-\s\)};
                        }
                    \foreach \s in {\lfz,...,\numexpr\fz+2\relax}
                        {
                            \node[dot] (\s) at ({360/\n * (-\s)+90}:\radius) {};
                        }
                    \foreach \s in {\lbz,...,\numexpr\bz+2\relax}
                        {
                            \node[dot] (nminus\s) at ({360/\n * (\s)+90}:\radius) {};
                        }
                    \foreach \s in {0,...,\numexpr\fz\relax}
                        {
                            \node[ddot] (disp\s) at ({360/\n * (-\s)+90}:\radius) {};
                        }
                    \foreach \s in {1,...,\numexpr\bz\relax}
                        {
                            \node[ddot] (disp\n-\s) at ({360/\n * (\s)+90}:\radius) {};
                        }
                    \node[dot] (halfn) at ({360/\n * (0)-90}:\radius) {};
                    \node[draw=none] (labelhalfn) at ({360/\n * (0)-90}:\radiustwo) {\small\(n\)};
                    \foreach \s in {1,...,\numexpr\lfh\relax}
                        {
                            \node[dot] (halfnplus\s) at ({360/\n * (-\s)-90}:\radius) {};
                            \node[draw=none] (labelhalfnplus\s) at ({360/\n * (-\s)-90}:\radiustwo) {\small\(n+\s\)};
                        }
                    \foreach \s in {1,...,\numexpr\lbh\relax}
                        {
                            \node[dot] (halfnminus\s) at ({360/\n * (\s)-90}:\radius) {};
                            \node[draw=none] (labelhalfnminus\s) at ({360/\n * (\s)-90}:\radiustwo) {\small\(n-\s\)};
                        }
                    \foreach \s in {\lfh,...,\numexpr\fh+2\relax}
                        {
                            \node[dot] (halfnplus\s) at ({360/\n * (-\s)-90}:\radius) {};
                        }
                    \foreach \s in {\lbh,...,\numexpr\bh+2\relax}
                        {
                            \node[dot] (halfnminus\s) at ({360/\n * (\s)-90}:\radius) {};
                        }
                    \foreach \s in {0,...,\numexpr\fh\relax}
                        {
                            \node[ddot] (disphalfnplus\s) at ({360/\n * (-\s)-90}:\radius) {};
                        }
                    \foreach \s in {1,...,\numexpr\bh\relax}
                        {
                            \node[ddot] (disphalfnminus\s) at ({360/\n * (\s)-90}:\radius) {};
                        }

                    \begin{scope}[on background layer]
                        \foreach \s/\y in  {nminus2/halfnminus2}
                            {
                                \draw [circF1]  (\s) -- (\y);

                            }
                        \foreach \s/\y in  {0/2,1/3,4/6,5/7, halfnminus11/halfnminus9,halfnminus10/halfnminus8,halfnminus7/halfnminus5,halfnminus6/halfnminus4,halfnminus3/halfnminus1,halfn/halfnplus2,halfnplus1/halfnplus3,halfnplus4/halfnplus6,halfnplus5/halfnplus7,nminus11/nminus9,nminus10/nminus8,nminus7/nminus5,nminus6/nminus4,nminus3/nminus1}
                            {
                                \draw [circF1]  (\s) to[bend left=-45] (\y);

                            }
                        \foreach \s/\y in {0/halfn,1/halfnplus1,2/halfnplus2,3/halfnplus3,4/halfnplus4,5/halfnplus5, nminus1/halfnminus1,nminus3/halfnminus3,nminus5/halfnminus5,nminus7/halfnminus7,nminus9/halfnminus9,nminus10/halfnminus10,nminus11/halfnminus11}
                            {
                                \draw [circF2]  (\s) -- (\y);
                            }
                        \foreach \s/\y in  {halfnminus8/halfnminus6,halfnminus4/halfnminus2,nminus8/nminus6,nminus4/nminus2}
                            {
                                \draw [circF2]  (\s) to[bend left=-45] (\y);

                            }
                        \foreach \s/\y in {nminus8/halfnminus8,nminus6/halfnminus6,nminus4/halfnminus4}
                            {
                                \draw [circF3]  (\s) -- (\y);

                            }
                        \foreach \s/\y in  {nminus13/nminus11,nminus12/nminus10,nminus9/nminus7,nminus5/nminus3,nminus2/0,nminus1/1,2/4,3/5,halfnminus13/halfnminus11,halfnminus12/halfnminus10,halfnminus9/halfnminus7,halfnminus5/halfnminus3,halfnminus2/halfn,halfnminus1/halfnplus1,halfnplus2/halfnplus4,halfnplus3/halfnplus5}
                            {
                                \draw [circF3]  (\s) to[bend left=-45] (\y);

                            }
                    \end{scope}
                    \begin{scope}[node distance=1cm, every node/.style={font=\sffamily}, align=left]
                        \matrix [above right=of current bounding box.north east, yshift=-0.5cm,xshift=-0.25cm,anchor=north west, nodes={inner sep=0pt}, row sep=0.15cm] {
                        \node [label=right:{\(F_a\)}] (legend1) {};
                        \draw [circF1] ([xshift=-1.1cm]legend1.west) -- ([xshift=-0.1cm]legend1.west);
                        \\
                        \node [label=right:{\(F_b\)}] (legend2) {};
                        \draw [circF2] ([xshift=-1.1cm]legend2.west) -- ([xshift=-0.1cm]legend2.west);
                        \\
                        \node [label=right:{\(F_c\)}] (legend3) {};
                        \draw [circF3] ([xshift=-1.1cm]legend3.west) -- ([xshift=-0.1cm]legend3.west);
                        \\
                        };
                    \end{scope}
                \end{tikzpicture}
            }
        \end{subfigure}
        \caption{Construction of 3-B1F of \(Circ(2n,\left\lbrace 2,n \right\rbrace)\) for \(n\equiv 3 \pmod 4\).}
        \label{Fig: 6 mod 8 General Construction 3-B1F Circ(n,{2,n/2})}
    \end{figure}

    Similar to before, we will show that \(\mathcal{F}_\beta\) is a 3-B1F by showing that the union of each pair of 1-factors is of a different type.

    The union $F_a\cup F_b$ is of type \(\left[10,4^{\tfrac{2n-10}{4} }\right]\)
    with cycles
    \begin{gather*}
        \left(2n-2,2n-4,2n-6,2n-8,2n-10, n-10, n-8, n-6, n-4,n-2 \right)\text{ and }\\
        \left(x,x+2,n+x+2,n+x\right) \text{ for } x\in \left\lbrace 0, 4, \dots, n-3  \right\rbrace \cup \left\lbrace 1, 5, \dots n-14  \right\rbrace.
    \end{gather*}
    The union $F_a\cup F_c$ is of type \(\left[2n-4,4\right]\) with cycles
    \begin{gather*}
        \left(0, 2, 4, \dots, 2n-8, n-8, n-10, \dots, 1, 2n-1, 2n-3, \dots, n-2, 2n-\right)\text{ and }\\
        \left(2n-6,2n-4,n-4,n-6\right).
    \end{gather*}
    The union $F_b\cup F_c$ is of type \(\left[6,4^{\tfrac{2n-6}{4} }\right]\)
    with cycles
    \begin{gather*}
        \left(0,2n-2,2n-4,n-4,n-2, n\right)\text{ and } \\ \left(x,x+2, x+2+n, x+n\right)\text{ for } x\in \left\lbrace 2, 6, \dots, n-1 \right\rbrace \cup \left\lbrace 3, 7, \dots, n-8  \right\rbrace.
    \end{gather*}
    Thus, $\mathcal{F}_\beta$ is a 3-B1F with types \(\left(\left[10,4^{\tfrac{2n-10}{4} }\right],\left[2n-4,4\right],\left[6,4^{\tfrac{2n-6}{4} }\right]\right)\).

    Therefore the connected 3-regular circulant graph, \(Circ(2n,\left\lbrace 2,n  \right\rbrace)\), admits a 3-B1F if and only if \(n \geq 9\).
\end{proof}

\section{B1Fs of Connected 4-Regular Circulant Graphs}
\label{Subsec:B1Fs of 4-Regular Circulant Graphs}
Note that an \(m\)-B1F of a 4-regular graph requires \(m\in \left\lbrace 1,2,3,6 \right\rbrace\).
We will first consider circulant graphs with connection set \(\left\lbrace 1,2 \right\rbrace\) as we will see that this connection set makes 1-factorisations quite restricted. Our work in \ref{Subsubsec:B1Fs of Circ(n,{1,2})} together with Herke's work in \cite{SaraPhDThesis} allows us to completely characterise the 4-regular circulant graphs with connection set \(\left\lbrace 1,2 \right\rbrace\) that admit an \(m\)-B1F (see Theorem \ref{Thm: m-B1Fs of Circ(n,{1,2}) existence}). Afterwards we look towards finding constructions of families of B1Fs of 4-regular circulant graphs with different connection sets.

\subsection{B1Fs of \(\mathbf{Circ(2n,\left\lbrace 1,2 \right\rbrace)}\)}
\label{Subsubsec:B1Fs of Circ(n,{1,2})}
Herke \cite{SaraPhDThesis} showed that \(Circ(2n,\left\lbrace 1,2 \right\rbrace)\) admits a 1-B1F if and only if \(n\in \left\lbrace  2,3 \right\rbrace\), so we will focus on \(m\)-B1Fs where \(m\in \left\lbrace  2,3,6 \right\rbrace\).
We first need to introduce a number of definitions.
A \emph{\(1\)-edge} (\emph{\(2\)-edge}) of \(Circ(2n,\left\lbrace 1,2 \right\rbrace)\) is an edge of the form $\left\{v,v+1\right\}$ \(\left(\left\lbrace v,v+2 \right\rbrace\right)\) where addition is performed modulo \(2n\).
Further, the \emph{out edge} of a vertex $v$ is the edge \(\left\lbrace v,v+\epsilon \right\rbrace\) for some positive integer \(\epsilon\), similarly the \emph{in edge} of a vertex \(v\) is the edge \(\left\lbrace v-\epsilon,v \right\rbrace\) for some positive integer \(\epsilon\).
We also adapt the ideas of \(k\)-configurations from \cite{Herke2013P1FCirc}. The \emph{configuration} \(C_e\) of the 1-edge \(e=\left\lbrace v,v+1 \right\rbrace\) is the set of edges \(\left\lbrace \left\lbrace v-1,v+1 \right\rbrace, \left\lbrace v,v+1 \right\rbrace,\left\lbrace v,v+2 \right\rbrace \right\rbrace\). Intuitively, if we imagine the vertices of \(Circ(2n,\left\lbrace  1,2 \right\rbrace)\) arranged along a circle in the natural order, then \(C_e\) is the set of all the edges that intersect with an imagined line bisecting \(e\). Given some edge colouring of the edges of \(Circ(2n,\left\lbrace 1,2 \right\rbrace)\), a configuration of a 1-edge is a \emph{\(k\)-configuration} if \(k\) distinct colours are used on its edges.
The \emph{out edges} (\emph{in edges}) of a configuration \(C_e\) are the out edges (in edges) of the vertices of \(e\). We say that two configurations \(C_e\) and \(C_{e'}\) are adjacent if the edges \(e\) and \(e'\) share a vertex.

To set up the observations and lemmas that follow, we let \(\mathcal{F}=\left\lbrace R,G,B,Y \right\rbrace\) be an arbitrary 1-factorisation of \(Circ(2n,\left\lbrace 1,2 \right\rbrace)\) for even \(n \geq  3\), and we colour the edges of the 1-factors \(R,G,B,Y\) with red, green, blue, and yellow, respectively.

\begin{observation}
    The 1-edges of \(\mathcal{F}\) must be 2-configurations or 3-configurations. Further, the 2-edges of a 2-configuration must be the same colour.
\end{observation}

\begin{lemma}\label{Lem: 2-configurations surround 3-configurations}
    If \(\mathcal{F}\) contains a 3-configuration \(C_e\), then the two configurations that are adjacent to \(C_e\) will be 2-configurations. Further, the 1-edges of these 2-configurations will be the same colour.
\end{lemma}

\begin{figure}[ht!]
    \begin{center}
        \begin{tikzpicture}[dot/.style={circle, fill=black, inner sep=0pt, minimum size=1pt},ddot/.style={circle,fill=black,inner sep=0pt,minimum size=10pt},
                lbl/.style={font=\Large\bfseries} ]
            \def\radius{5}
            \def\startAngle{90}
            \def\angleInc{20}

            \node[dot,label={[lbl,label distance=2mm]\startAngle+\angleInc*2.5:n-1}] (noden) at (\startAngle+\angleInc*2.5:\radius) {};

            \node[dot,label={[lbl,label distance=2mm]\startAngle+\angleInc*1.5:0}] (node0) at (\startAngle+\angleInc*1.5:\radius) {};

            \node[dot,label={[lbl,label distance=2mm]\startAngle+\angleInc*0.5:1}] (node1) at (\startAngle+\angleInc*0.5:\radius) {};

            \node[dot,label={[lbl,label distance=2mm]\startAngle-\angleInc*0.5:2}] (node2) at (\startAngle-\angleInc*0.5:\radius) {};

            \node[dot,label={[lbl,label distance=2mm]\startAngle-\angleInc*1.5:3}] (node3) at (\startAngle-\angleInc*1.5:\radius) {};

            \node[dot,label={[lbl,label distance=2mm]\startAngle-\angleInc*2.5:4}] (node4) at (\startAngle-\angleInc*2.5:\radius) {};

            \node[ddot] (dispdnoden) at (\startAngle+\angleInc*2.5:\radius) {};

            \node[ddot] (dispdnode0) at (\startAngle+\angleInc*1.5:\radius) {};

            \node[ddot] (dispdnode1) at (\startAngle+\angleInc*0.5:\radius) {};

            \node[ddot] (dispdnode2) at (\startAngle-\angleInc*0.5:\radius) {};

            \node[ddot] (dispdnode3) at (\startAngle-\angleInc*1.5:\radius) {};

            \node[ddot] (dispdnode4) at (\startAngle-\angleInc*2.5:\radius) {};
            \begin{scope}[on background layer]
                \draw [circ_darkyellow] (node0) -- (node1);
                \draw [circ_darkred] (node1) -- (node2);
                \draw [circ_darkblue] (noden) to [bend right=15] (node1);
                \draw [circ_darkblue] (node0) to [bend right=15] (node2);
                \draw [circ_darkyellow] (node2) -- (node3);
                \draw [circ_darkgreen] (node1) to [bend right=15] (node3);
                \draw [circ_darkgreen] (node2) to [bend right=15] (node4);

            \end{scope}
        \end{tikzpicture}
        \caption{The adjacent configurations of a 3-configuration.}\label{Fig: 3-config neighbours}
    \end{center}
\end{figure}
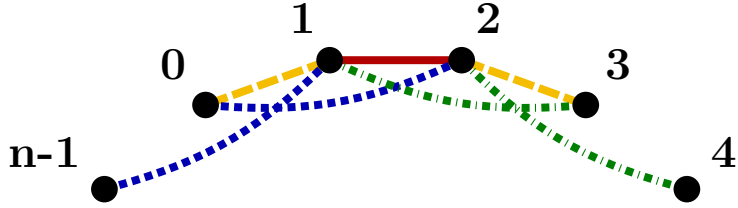

\begin{proof}
    Without loss of generality, we can assume that the configuration \(C_{\left\lbrace 1,2 \right\rbrace}\) is a 3-configuration with edges \(\left\lbrace 0,2 \right\rbrace,\left\lbrace 1,2 \right\rbrace,\left\lbrace 1,3 \right\rbrace\) coloured blue, red, and green, respectively (see Figure \ref{Fig: 3-config neighbours}).
    Since the colours represent a 1-factorisation, the edge \(\left\lbrace 0,1 \right\rbrace\) must be yellow and the edge \(\left\lbrace n-1,1 \right\rbrace\) must be blue. Thus, the configuration \(C_{\left\lbrace 0,1 \right\rbrace}\) must be a 2-configuration.
    Similarly, \(\left\lbrace 2,3 \right\rbrace\) must be yellow and \(\left\lbrace 2,4 \right\rbrace\) must be green, making the configuration \(C_{\left\lbrace 2,3 \right\rbrace}\) a 2-configuration. Thus, the adjacent configurations of a 3-configurations are 2-configurations with their 1-edges having the same colour.
\end{proof}

\begin{lemma}
    If \(\mathcal{F}\) contains a 2-configuration \(C_e\), then the configurations that are adjacent to \(C_e\) will be 3-configurations.
\end{lemma}

\begin{figure}[ht!]
    \begin{center}
        \begin{tikzpicture}[dot/.style={circle, fill=black, inner sep=0pt, minimum size=1pt},ddot/.style={circle,fill=black,inner sep=0pt,minimum size=10pt},
                lbl/.style={font=\Large\bfseries} ]
            \def\radius{5}
            \def\startAngle{90}
            \def\angleInc{20}


            \node[dot,label={[lbl,label distance=2mm]\startAngle+\angleInc*1.5:0}] (node0) at (\startAngle+\angleInc*1.5:\radius) {};

            \node[dot,label={[lbl,label distance=2mm]\startAngle+\angleInc*0.5:1}] (node1) at (\startAngle+\angleInc*0.5:\radius) {};

            \node[dot,label={[lbl,label distance=2mm]\startAngle-\angleInc*0.5:2}] (node2) at (\startAngle-\angleInc*0.5:\radius) {};

            \node[dot,label={[lbl,label distance=2mm]\startAngle-\angleInc*1.5:3}] (node3) at (\startAngle-\angleInc*1.5:\radius) {};

            \node[dot,label={[lbl,label distance=2mm]\startAngle-\angleInc*2.5:4}] (node4) at (\startAngle-\angleInc*2.5:\radius) {};

            \node[ddot] (dispDnode0) at (\startAngle+\angleInc*1.5:\radius) {};

            \node[ddot] (dispDnode1) at (\startAngle+\angleInc*0.5:\radius) {};

            \node[ddot] (dispDnode2) at (\startAngle-\angleInc*0.5:\radius) {};

            \node[ddot] (dispDnode3) at (\startAngle-\angleInc*1.5:\radius) {};

            \node[ddot] (dispDnode4) at (\startAngle-\angleInc*2.5:\radius) {};
            \begin{scope}[on background layer]
                \draw [darkgreen, dash dot, line width=3pt] (node1) -- (node2);
                \draw [darkred, line width=3pt] (node0) to [bend right=15] (node2);
                \draw [gray,dash pattern=on 8pt off 2pt, line width=3pt] (node2) -- (node3);
                \draw [darkred, line width=3pt] (node1) to [bend right=15] (node3);
                \draw [gray, dash pattern=on 8pt off 2pt, line width=3pt] (node2) to [bend right=15] (node4);

            \end{scope}
        \end{tikzpicture}
        \caption{The adjacent configuration of a 2-configuration.}\label{Fig: 2-config neighbours}
    \end{center}
\end{figure}
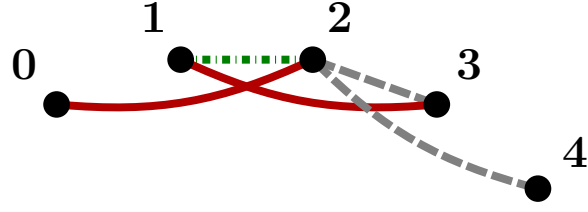

\begin{proof}
    Without loss of generality, we can assume that the configuration \(C_{\left\lbrace 1,2 \right\rbrace}\) is a 2-configuration with edges \( \left\lbrace 0,2 \right\rbrace,\left\lbrace 1,2 \right\rbrace,\left\lbrace 1,3 \right\rbrace \) coloured red, green, and red, respectively (see Figure \ref{Fig: 2-config neighbours}).
    Consider the configuration \(C_{\left\lbrace 2,3 \right\rbrace}\). Clearly the edges \(\left\lbrace 2,3 \right\rbrace \) and \(\left\lbrace 2,4 \right\rbrace\) must have different colours, neither of which can be red, so \(C_{\left\lbrace 2,3 \right\rbrace}\) must be a 3-configuration. A similar argument shows \(C_{\left\lbrace 0,1 \right\rbrace}\) is also a 3-configuration.
\end{proof}

\begin{corollary} \label{alternating configs}
    \(\mathcal{F}\) must contain alternating 2-configurations and 3-configurations.
\end{corollary}

\begin{figure}[ht!]
    \centering
    \begin{subfigure}[b]{0.4\textwidth}
        \begin{center}
            \definecolor{darkyellow}{rgb}{0.960938, 0.742188, 0}
            \definecolor{darkred}{rgb}{0.7, 0, 0}
            \definecolor{darkgreen}{rgb}{0, 0.5, 0}
            \definecolor{darkblue}{rgb}{0, 0, 0.7}
            \begin{tikzpicture}
                \def\radius{4}
                \def\startAngle{90}
                \def\angleInc{25}

                \node[circle,fill=black,inner sep=0pt,minimum size=1pt] (node0) at (\startAngle+\angleInc*1.5:\radius) {};
                \node[circle,fill=black,inner sep=0pt,minimum size=1pt] (node1) at (\startAngle+\angleInc*0.5:\radius) {};
                \node[circle,fill=black,inner sep=0pt,minimum size=1pt] (node2) at (\startAngle-\angleInc*0.5:\radius) {};
                \node[circle,fill=black,inner sep=0pt,minimum size=1pt] (node3) at (\startAngle-\angleInc*1.5:\radius) {};
                \node[circle,fill=black,inner sep=0pt,minimum size=1pt] (node4) at (\startAngle-\angleInc*2.5:\radius) {};
                \node[circle,fill=black,inner sep=0pt,minimum size=10pt] (Drawnode0) at (\startAngle+\angleInc*1.5:\radius) {};
                \node[circle,fill=black,inner sep=0pt,minimum size=10pt] (Drawnode1) at (\startAngle+\angleInc*0.5:\radius) {};
                \node[circle,fill=black,inner sep=0pt,minimum size=10pt] (Drawnode2) at (\startAngle-\angleInc*0.5:\radius) {};
                \node[circle,fill=black,inner sep=0pt,minimum size=10pt] (Drawnode3) at (\startAngle-\angleInc*1.5:\radius) {};
                \node[circle,fill=black,inner sep=0pt,minimum size=10pt] (Drawnode4) at (\startAngle-\angleInc*2.5:\radius) {};
                \begin{scope}[on background layer]
                    \draw [darkyellow, line width=3pt] (node1) -- (node2);
                    \draw [darkblue, dotted, line width=3pt] (node0) to [bend right=15] (node2);
                    \draw [yellow, line width=8pt,draw opacity = 0.75] (node1) to [bend right=15] (node3);
                    \draw [yellow, line width=8pt,draw opacity = 0.75] (node2) to [bend right=15] (node4);
                    \draw [darkblue, dotted, line width=3pt] (node1) to [bend right=15] (node3);
                    \draw [darkred, dash pattern=on 8pt off 2pt, line width=3pt] (node2) to [bend right=15] (node4);
                \end{scope}
            \end{tikzpicture}
            \caption{Out edges of a 2-configuration.}
        \end{center}
    \end{subfigure}
    \hfil
    \begin{subfigure}[b]{0.4\textwidth}
        \begin{center}
            \begin{tikzpicture}
                \def\radius{4}
                \def\startAngle{90}
                \def\angleInc{25}

                \node[circle,fill=black,inner sep=0pt,minimum size=1pt] (node0) at (\startAngle+\angleInc*1.5:\radius) {};
                \node[circle,fill=black,inner sep=0pt,minimum size=1pt] (node1) at (\startAngle+\angleInc*0.5:\radius) {};
                \node[circle,fill=black,inner sep=0pt,minimum size=1pt] (node2) at (\startAngle-\angleInc*0.5:\radius) {};
                \node[circle,fill=black,inner sep=0pt,minimum size=1pt] (node3) at (\startAngle-\angleInc*1.5:\radius) {};
                \node[circle,fill=black,inner sep=0pt,minimum size=1pt] (node4) at (\startAngle-\angleInc*2.5:\radius) {};
                \node[circle,fill=black,inner sep=0pt,minimum size=10pt] (Drawnode0) at (\startAngle+\angleInc*1.5:\radius) {};
                \node[circle,fill=black,inner sep=0pt,minimum size=10pt] (Drawnode1) at (\startAngle+\angleInc*0.5:\radius) {};
                \node[circle,fill=black,inner sep=0pt,minimum size=10pt] (Drawnode2) at (\startAngle-\angleInc*0.5:\radius) {};
                \node[circle,fill=black,inner sep=0pt,minimum size=10pt] (Drawnode3) at (\startAngle-\angleInc*1.5:\radius) {};
                \node[circle,fill=black,inner sep=0pt,minimum size=10pt] (Drawnode4) at (\startAngle-\angleInc*2.5:\radius) {};

                \begin{scope}[on background layer]
                    \draw [darkred, line width=3pt] (node1) -- (node2);
                    \draw [darkblue, dotted, line width=3pt] (node0) to [bend right=15] (node2);
                    \draw [yellow, line width=8pt,draw opacity = 0.75] (node1) to [bend right=15] (node3);
                    \draw [yellow, line width=8pt,draw opacity = 0.75] (node2) to [bend right=15] (node4);
                    \draw [darkgreen, dash pattern=on 8pt off 2pt, line width=3pt] (node1) to [bend right=15] (node3);
                    \draw [darkgreen, dash pattern=on 8pt off 2pt, line width=3pt] (node2) to [bend right=15] (node4);
                \end{scope}
            \end{tikzpicture}
            \caption{Out edges of a 3-configuration.}
        \end{center}
    \end{subfigure}
    \caption{The out 2-edges of the two possible configuration types.}
    \label{Fig:out edges of configs}
\end{figure}
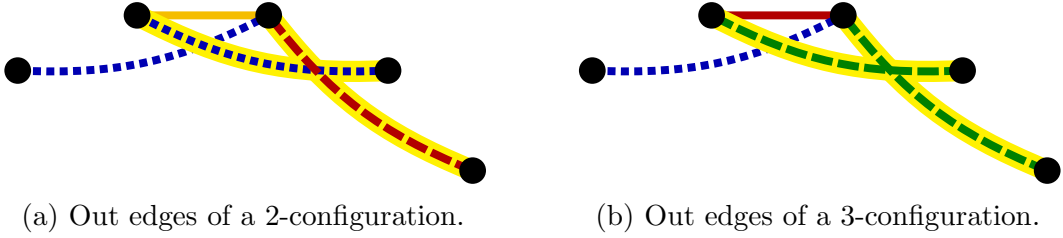

\begin{observation} \label{Obs: Config outedges}
    The out 2-edges of a 2-configuration are different colours, and the out 2-edges of a 3-configuration are the same colour. (See Figure \ref{Fig:out edges of configs} for an example.)
\end{observation}

Combining Corollary \ref{alternating configs} and Observation \ref{Obs: Config outedges} we have:
\begin{observation} \label{2-edge in pairs}
    There are an even number of 2-edges belonging to each 1-factor of \(\mathcal{F}\).
\end{observation}

\begin{lemma} \label{yellow factor is base edges}
    At least one 1-factor of \(\mathcal{F}\) must consist entirely of 1-edges.
\end{lemma}

\begin{proof}
    Using Corollary \ref{alternating configs} we know that \(\mathcal{F}\) must contain an alternating sequence of 2-configurations and 3-configurations. From Lemma \ref{Lem: 2-configurations surround 3-configurations} it follows that the 1-edges of the 2-configurations must all be the same colour.
\end{proof}

We can now, without loss of generality, let \(Y\) be the 1-factor consisting only of 1-edges corresponding to 2-configurations.

\begin{lemma} \label{Yellow union is Ham}
    The union of Y with any of R, G, or B, is a Hamilton cycle.
\end{lemma}
\begin{proof}
    Let \(X\in \mathcal{F}\setminus Y\) be a non-yellow 1-factor, and consider the union \(X\cup Y\).
    Without loss of generality, let the yellow edges be \(\left\lbrace 0,1 \right\rbrace,\left\lbrace 2,3 \right\rbrace, \dots , \left\lbrace n-2,n-1 \right\rbrace \) and consider some vertex $2k$ with an in edge in \(X\) and an out 1-edge in \(Y\). Depending on the edges of \(X\) this vertex lies on the path \(\left[2k,2k+1,2k+2\right]\) or the path \(\left[2k,2k+1,2k+3,2k+2,2k+4\right]\) in \(X \cup Y\)
    (see Figure \ref{Fig:two path cases yellow ham}). We note here that in either case the vertex set of the path is a set of consecutive vertices and all internal vertices lie between the endpoints. As this is the case for every even-labelled vertex with an in edge in \(X\) and an out 1-edge in \(Y\), and since these paths end on even-labelled vertices with an in edge in \(X\) and an out 1-edge in \(Y\), we repeat this argument to see that \(X\cup Y\) consists of the union of these internally-disjoint paths that together form a Hamilton cycle.
\end{proof}

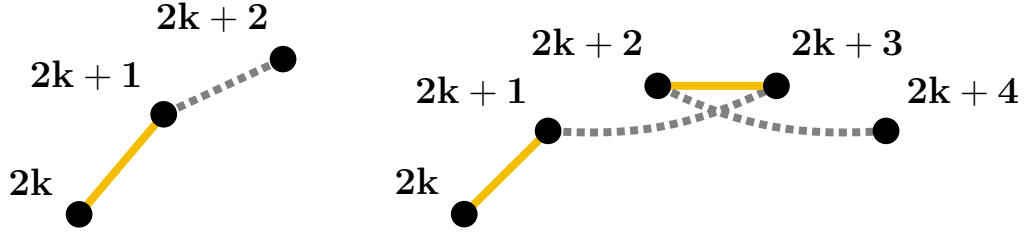
\begin{figure}[ht!]
    \centering
    \begin{subfigure}[b]{0.1\textwidth}
        \begin{center}
            \definecolor{darkyellow}{rgb}{0.960938, 0.742188, 0}
            \begin{tikzpicture}[dot/.style={circle, fill=black, inner sep=0pt, minimum size=1pt},ddot/.style={circle,fill=black,inner sep=0pt,minimum size=10pt},
                    lbl/.style={font=\large\bfseries} ]
                \def\radius{4}
                \def\startAngle{90}
                \def\angleInc{25}

                \node[dot,label={[lbl,label distance=2mm]\startAngle+\angleInc*2.5:\(\mathbf{2k}\)}] (noden) at (\startAngle+\angleInc*2.5:\radius) {};

                \node[dot,label={[lbl,label distance=2mm]\startAngle+\angleInc*1.5:\(\mathbf{2k+1}\)}] (node0) at (\startAngle+\angleInc*1.5:\radius) {};

                \node[dot,label={[lbl,label distance=2mm]\startAngle+\angleInc*0.5:\(\mathbf{2k+2}\)}] (node1) at (\startAngle+\angleInc*0.5:\radius) {};




                \node[ddot] (dispdnoden) at (\startAngle+\angleInc*2.5:\radius) {};

                \node[ddot] (dispdnode0) at (\startAngle+\angleInc*1.5:\radius) {};

                \node[ddot] (dispdnode1) at (\startAngle+\angleInc*0.5:\radius) {};




                \begin{scope}[on background layer]
                    \draw [darkyellow, line width=3pt] (noden) -- (node0);
                    \draw [gray, dotted, line width=3pt] (node0) -- (node1);
                \end{scope}
            \end{tikzpicture}
        \end{center}
    \end{subfigure}
    \hspace{0.20\textwidth}
    \begin{subfigure}[b]{0.4\textwidth}
        \begin{center}
            \begin{tikzpicture}[dot/.style={circle, fill=black, inner sep=0pt, minimum size=1pt},ddot/.style={circle,fill=black,inner sep=0pt,minimum size=10pt},
                    lbl/.style={font=\large\bfseries} ]
                \def\radius{4}
                \def\startAngle{90}
                \def\angleInc{22.5}

                \node[dot,label={[lbl,label distance=2mm]\startAngle+\angleInc*2.5:\(\mathbf{2k}\)}] (noden) at (\startAngle+\angleInc*2.5:\radius) {};

                \node[dot,label={[lbl,label distance=2mm]\startAngle+\angleInc*1.5:\(\mathbf{2k+1}\)}] (node0) at (\startAngle+\angleInc*1.5:\radius) {};

                \node[dot,label={[lbl,label distance=2mm]\startAngle+\angleInc*0.5:\(\mathbf{2k+2}\)}] (node1) at (\startAngle+\angleInc*0.5:\radius) {};

                \node[dot,label={[lbl,label distance=2mm]\startAngle-\angleInc*0.5:\(\mathbf{2k+3}\)}] (node2) at (\startAngle-\angleInc*0.5:\radius) {};

                \node[dot,label={[lbl,label distance=2mm]\startAngle-\angleInc*1.5:\(\mathbf{2k+4}\)}] (node3) at (\startAngle-\angleInc*1.5:\radius) {};


                \node[ddot] (dispdnoden) at (\startAngle+\angleInc*2.5:\radius) {};

                \node[ddot] (dispdnode0) at (\startAngle+\angleInc*1.5:\radius) {};

                \node[ddot] (dispdnode1) at (\startAngle+\angleInc*0.5:\radius) {};

                \node[ddot] (dispdnode2) at (\startAngle-\angleInc*0.5:\radius) {};

                \node[ddot] (dispdnode3) at (\startAngle-\angleInc*1.5:\radius) {};


                \begin{scope}[on background layer]
                    \draw [darkyellow, line width=3pt] (noden) -- (node0);
                    \draw [gray, dotted, line width=3pt] (node0) to[bend right =15] (node2);
                    \draw [darkyellow, line width=3pt] (node2) -- (node1);
                    \draw [gray, dotted, line width=3pt] (node1) to[bend right =15] (node3);

                \end{scope}
            \end{tikzpicture}
        \end{center}
    \end{subfigure}
    \caption{The two cases of path segments in \(X\cup Y\).}
    \label{Fig:two path cases yellow ham}
\end{figure}

\begin{corollary}
    If \(\mathcal{F}=\left\lbrace R,G,B,Y \right\rbrace\) is an B1F, it must be a 1-B1F or 2-B1F. Further, \(R\cup G, R\cup B,\) and \(G\cup B\) must be of the same type.
\end{corollary}

\begin{lemma}
    A cycle in \(R\cup G, R\cup B,\) or \(G\cup B\), will contain either zero or two 1-edges.
\end{lemma}

\begin{proof}
    Consider two distinct non-yellow 1-factors; without loss of generality, let them be \(R\) and \(G\) and consider \(R \cup G\).
    We first show that a cycle with zero 1-edges can exist in \(R\cup G\).
    If $n\equiv 0 \pmod 4$, then one possibility of a cycle in \(R \cup G\) is a cycle of length $\tfrac{n}{2}$ made from alternating 2-edges.

    We now show that if there exists a 1-edge in a cycle of \(R \cup G\), then that cycle must contain exactly two 1-edges.
    Without loss of generality, suppose there is a cycle in \(R\cup G\) that contains the green 1-edge \(\left\lbrace 1,2 \right\rbrace\) whose out 2-edges are red (recalling that \(C_{\left\lbrace  1,2 \right\rbrace}\) must be a 3-configuration in \(\mathcal{F}\) and that the out 2-edges of a 3-configuration are the same colour). From Corollary \ref{alternating configs} we know that \(C_{\left\lbrace 3,4 \right\rbrace}\) is also a 3-configuration, thus either \(\left\lbrace 3,4 \right\rbrace \in R \cup G\) in which case the edge \(\left\lbrace 1,2 \right\rbrace\) belongs to a 4-cycle in \(R\cup G\) (which contains exactly two 1-edges), or the out 2-edges of \(C_{\left\lbrace\,3,4 \right\rbrace}\) belong to \(R\cup G\) in which case the edge \(\left\lbrace  1,2 \right\rbrace\) belongs to a cycle of length at least 6. If the latter is the case we use similar logic to conclude that either \(\left\lbrace  1,2 \right\rbrace\) belongs to a 6-cycle with exactly two 1-edges, or \(\left\lbrace  1,2 \right\rbrace\) belongs to a cycle of length at least 8. If we continue this logic, we can see that the edge \(\left\lbrace  1,2 \right\rbrace\) will always belong to an even length cycle that contains exactly two 1-edges.
\end{proof}
We also note that if the union of two non-yellow 1-factors contains a cycle of length $\tfrac{n}{2}$ made from alternating 2-edges, then the vertices not on that cycle must lie on another cycle of length $\tfrac{n}{2}$ made from alternating 2-edges. This along with the previous lemma gives the following.
\begin{corollary} \label{Cor: num cycles is half 1-edges}
    Let \(X,Z\in \mathcal{F}\setminus Y\) be non-yellow 1-factors of \(\mathcal{F}\) and consider their union \(X \cup Z\).
    The number of cycles in \(X\cup Z\) is equal to half the number of 1-edges in the pair of 1-factors if there is at least one 1-edge, and equal to two if there are no 1-edges.
\end{corollary}

\begin{theorem}
    \(\mathcal{F}\) is not a 2-B1F of $Circ(2n,\{1,2\})$ when $n\equiv 1,2 \pmod 3, n\geq 5$.
\end{theorem}

\begin{proof}
    From Lemma \ref{Yellow union is Ham} we know that $R\cup Y,G\cup Y,B\cup Y$ all form Hamilton cycles, thus to show that there is no 2-B1F it suffices to show that $B \cup G$ and $R \cup G$ do not share the same cycle structure.

    Let $x_R,x_G,x_B$ $(y_R,y_G,y_B)$ denote the number of 1-edges (2-edges) in R, G, B respectively. Without loss of generality let $x_B\leq x_G \leq x_R$.
    As \(2n \not\equiv 0 \pmod 3\), we note that $B\cup G$ has fewer 1-edges than $R\cup G$. We also know that since $x_i+y_i=n , \forall i \in \{R,G,B\}$ and \(y_i\) is even for all \(i \in \left\lbrace R,G,B \right\rbrace\) by Observation \ref{2-edge in pairs}, $x_i\equiv n \pmod 2$. It follows that \(B\cup G\) and \(R\cup G\) will both have an even number of 1-edges, so let these counts be $2j_{B \cup G}$ and $2j_{R \cup G}$, respectively, where \(j_{B \cup G}\) and \(j_{R \cup G}\) are integers.

    If $2j_{B \cup G}>0$, then by Corollary \ref{Cor: num cycles is half 1-edges} we have that the number of cycles in $B \cup G$ is $j_{B \cup G}$ and the number of cycles in $R \cup G$ is $j_{R \cup G}$.
    As $j_{B \cup G} < j_{R \cup G}$, $B\cup G$ and $R\cup G$ cannot have the same cycle structure.
    If $2j_{B \cup G}=0$, then $2j_{R \cup G}=n$, thus \(2j_{R \cup G} \geq 8\).
    From Corollary \ref{Cor: num cycles is half 1-edges}, we can see that there are two cycles in $B\cup G$ and at least four cycles in $R\cup G$.
    Hence, $B\cup G$ and $R\cup G$ cannot have the same cycle structure.

    Thus, this 1-factorisation is not an B1F. As the 1-factorisation was arbitrary, $Circ(2n,\{1,2\})$ does not admit a B1F when $n\equiv 1,2 \pmod 3, n\geq 5$.
\end{proof}

We now show that for \(n \geq 3  \), if \(n=4\) or \(n \equiv 0 \pmod 3 \) then \(Circ(2n,\left\lbrace 1,2 \right\rbrace)\) admits a 2-B1F.

\begin{lemma}
    There exists a 2-B1F of \(Circ(8,\left\lbrace  1,2 \right\rbrace)\).
\end{lemma}
\begin{proof}
    Consider the 1-factorisation \(\mathcal{F}=\left\lbrace R,G,B,Y \right\rbrace\) defined as follows:
    \begin{align*}
        R & = \left\lbrace \left\lbrace 0,1 \right\rbrace, \left\lbrace 2,3 \right\rbrace,\left\lbrace 4,5 \right\rbrace,\left\lbrace 6,7 \right\rbrace \right\rbrace,  \\
        G & = \left\lbrace \left\lbrace 0,2 \right\rbrace,\left\lbrace 1,3 \right\rbrace,\left\lbrace 4,6 \right\rbrace,\left\lbrace 5,7 \right\rbrace \right\rbrace,   \\
        B & = \left\lbrace \left\lbrace 0,6 \right\rbrace, \left\lbrace 1,7 \right\rbrace, \left\lbrace 2,4 \right\rbrace,\left\lbrace 3,5 \right\rbrace \right\rbrace, \\
        Y & = \left\lbrace \left\lbrace 0,7 \right\rbrace,\left\lbrace 1,2 \right\rbrace,\left\lbrace 3,4 \right\rbrace,\left\lbrace 5,6 \right\rbrace \right\rbrace.
    \end{align*}
    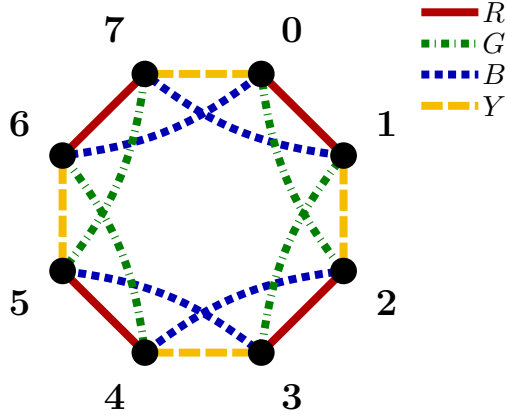
\begin{figure}
        \centering
        \begin{tikzpicture}[dot/.style={circle, fill=black, inner sep=0pt, minimum size=1pt},ddot/.style={circle,fill=black,inner sep=0pt,minimum size=10pt},
                lbl/.style={font=\large\bfseries} ]
            \def\radius{2}
            \def\startAngle{90}
            \def\angleInc{22.5}

            \foreach \s in {0,...,7}
                {
                    \node[dot,label={[lbl,label distance=2mm]{360/8 * (-\s)+67.5}:\(\mathbf{\s}\)}] (\s) at ({360/8 * (-\s)+67.5}:\radius) {{\tiny$\s$}};
                }
            \foreach \s in {0,...,7}
                {
                    \node[ddot] (dispdnode\s) at ({360/8 * (-\s)+67.5}:\radius) {{\tiny$\s$}};
                }

            \begin{scope}[on background layer]
                \draw [circ_darkred] (0) -- (1);
                \draw [circ_darkred] (2) -- (3);
                \draw [circ_darkred] (4) -- (5);
                \draw [circ_darkred] (6) -- (7);
                \draw [circ_darkyellow] (0) -- (7);
                \draw [circ_darkyellow] (1) -- (2);
                \draw [circ_darkyellow] (3) -- (4);
                \draw [circ_darkyellow] (5) -- (6);
                \draw [circ_darkblue] (6) to[bend right = 15] (0);
                \draw [circ_darkblue] (7) to[bend right = 15] (1);
                \draw [circ_darkblue] (2) to[bend right = 15] (4);
                \draw [circ_darkblue] (3) to[bend right = 15] (5);
                \draw [circ_darkgreen] (0) to[bend right = 15] (2);
                \draw [circ_darkgreen] (1) to[bend right = 15] (3);
                \draw [circ_darkgreen] (4) to[bend right = 15] (6);
                \draw [circ_darkgreen] (5) to[bend right = 15] (7);
            \end{scope}
            \begin{scope}[node distance=1cm, every node/.style={font=\sffamily}, align=left]
                \matrix [above right = of current bounding box.north east, yshift=-0.5cm,anchor=north, nodes={inner sep=0pt}, row sep=0.15cm] {
                \node [label=right:{\small\(R\)}] (legend2) {};
                \draw [circ_darkred] ([xshift=-0.75cm]legend2.west) -- (legend2.west);
                \\
                \node [label=right:{\small\(G\)}] (legend4) {};
                \draw [circ_darkgreen] ([xshift=-0.75cm]legend4.west) -- (legend4.west);
                \\
                \node [label=right:{\small\(B\)}] (legend3) {};
                \draw [circ_darkblue] ([xshift=-0.75cm]legend3.west) -- (legend3.west);
                \\
                \node [label=right:{\small\(Y\)}] (legend1) {};
                \draw [circ_darkyellow] ([xshift=-0.75cm]legend1.west) -- (legend1.west);
                \\
                };
            \end{scope}
        \end{tikzpicture}
        \caption{A 2-B1F of \(Circ(8,\left\lbrace 1,2 \right\rbrace)\).}\label{Fig: 2-B1F Circ(8,{1,2})}
    \end{figure}
    It is straightforward to confirm that \(\mathcal{F}\) is a 2-B1F with types \(\left(\left[8\right],\left[4^2\right]\right)\). (See Figure \ref{Fig: 2-B1F Circ(8,{1,2})}.)
\end{proof}

\begin{lemma} \label{Lem: Construction of 2-B1F of Circ(6a,{1,3})}
    For all \(n\equiv 0 \pmod 3\), \(n>3\), $Circ(2n,\{1,2\})$ admits a 2-B1F with types \(\left(\left[2n\right],\left[6^{\tfrac{2n}{6}}\right]\right)\).
\end{lemma}

\begin{proof}
    Let \(G=Circ(2n,\{1,2\})\) for some integer \(n>3\) where \(n\equiv 0\pmod 3\). We begin by constructing four 1-factors as follows
    \begin{align*}
        \text{R}	\,  = & \left\lbrace
        \left\lbrace
        x,x+1
        \right\rbrace
        : x \equiv 1\Mod{6}
        \right\rbrace
        \cup
        \left\lbrace
        \left\lbrace
        x,x+2
        \right\rbrace
        : x \equiv 3,4\Mod{6}
        \right\rbrace  ,              \\
        \text{G}	\,  = & \left\lbrace
        \left\lbrace
        x,x+1
        \right\rbrace
        : x \equiv 3\Mod{6}
        \right\rbrace
        \cup
        \left\lbrace
        \left\lbrace
        x,x+2
        \right\rbrace
        : x \equiv 0,5\Mod{6}
        \right\rbrace ,               \\
        \text{B}	  =   & \left\lbrace
        \left\lbrace
        x,x+1
        \right\rbrace
        : x \equiv 5\Mod{6}
        \right\rbrace
        \cup
        \left\lbrace
        \left\lbrace
        x,x+2
        \right\rbrace
        : x \equiv 1,2\Mod{6}
        \right\rbrace,                \\
        \text{Y}	  =   & \left\lbrace
        \left\lbrace
        x,x+1
        \right\rbrace
        : x \equiv 0\Mod{2}
        \right\rbrace.
    \end{align*}
    Clearly, \(\mathcal{F}=\left\lbrace R, G, B, Y \right\rbrace\) is a 1-factorisation of \(G\). Define a mapping \(\phi : V(G) \rightarrow V(G)\) where \(\phi(v)=v+2\) (addition is performed modulo \(2n\)). One can easily confirm that \(\phi(Y)=Y, \phi(R)=G, \phi(G)=B\), and \(\phi(B)=R\).

    From Lemma \ref{Yellow union is Ham} we know that \(R \cup Y, G \cup Y, \textrm{and }B \cup Y\) are all of type \(\left[n\right]\). Due to the symmetry of \(R,G,B\) under \(\phi\), the remaining unions of pairs are all isomorphic to each other. To determine the type, we will consider the union \(R \cup G\). It is straightforward to check that \(R \cup G\) is of type \(\left[6^{\tfrac{2n}{6}}\right]\) with cycles \[(x,x+2,x+4,x+5,x+3,x+1)\text{ for } x\equiv 3 \Mod 6.\] Thus, \(\mathcal{F}\) is a 2-B1F with types \(\left(\left[2n\right],\left[6^{\tfrac{2n}{6}}\right]\right)\).
\end{proof}

Recall that Herke \cite{SaraPhDThesis} showed that \(Circ(2n,\left\lbrace 1,2 \right\rbrace)\) admits a 1-B1F if and only if \(n\in \left\lbrace  2,3 \right\rbrace\). Combining this result with the above two lemmas lets us state the following theorem.

\begin{theorem}\label{Thm: m-B1Fs of Circ(n,{1,2}) existence}
    For \(n \geq 2 \), \(Circ(2n,\left\lbrace 1,2 \right\rbrace)\) admits an \(m\)-B1F if and only if
    \begin{enumerate}
        \item \(m=1\) and \(n\in \left\lbrace 2,3 \right\rbrace\); or
        \item \(m=2\) and either \(n=4\) or \(n\equiv 0 \pmod 3\).
    \end{enumerate}
\end{theorem}

\subsection{B1Fs of \(\mathbf{Circ(2n,\left\lbrace 1,3 \right\rbrace)}\)}
Herke \cite{SaraPhDThesis} showed that \(Circ(2n,\left\lbrace 1,3 \right\rbrace)\) admits a \(1\)-B1F if and only if \(n=4\), \(n>3\) and \(n\equiv 1 \pmod 2\), or \(n \geq 10 \) and \(n\equiv 0 \pmod 2\). We will give constructions of \(m\)-B1Fs of \(Circ(2n,\left\lbrace 1,3 \right\rbrace)\) for \(m\in \left\lbrace  2,3,6 \right\rbrace\). The proofs in this subsection will make use of the following definitions and observations.

We say that a 1-factor, \(F\), of a circulant graph has a \emph{gap} at vertices \((v,v+1)\) if \(F\) does not contain the edges \(\left\lbrace v,v+1 \right\rbrace,\left\lbrace v-2,v+1 \right\rbrace,\left\lbrace v,v+3 \right\rbrace,\) or \(\left\lbrace v-1,v+2 \right\rbrace\). We say that a circulant graph has \(z\) \emph{sequential gaps} between vertices \(v\) and \(v+z\) in the order \(F_1,F_2, \dots, F_z \) if there are \(z\) 1-factors, \(F_1,F_2, \dots, F_z \), that have gaps at \((v,v+1),(v+1,v+2), \dots, (v+z-1,v+z) \) respectively. For our coming construction we are interested in 1-factorisations with 4 sequential gaps; the following lemma will be helpful.

\begin{lemma} \label{Lem: Inner edges of 4 gap}
    Let \(\mathcal{F}=\left\lbrace R,G,B,Y \right\rbrace\) be a 1-factorisation of \(Circ(2n,\left\lbrace 1,3 \right\rbrace)\). If \(\mathcal{F}\) has 4 sequential gaps between vertices \(v\) and \(v+4\) in the order \(R,G,B,Y\), then \(\left\lbrace v,v+3 \right\rbrace\in Y \) and \(\left\lbrace v+1,v+4 \right\rbrace\in R\).
\end{lemma}

\begin{proof}
    Suppose that \(\mathcal{F}\) has 4 sequential gaps between vertices \(v\) and \(v+4\) in the order \(R,G,B,Y\) and a that the edge \(\left\lbrace v,v+3 \right\rbrace \not\in Y\). As \(Y\) is a 1-factor and has a gap at \((v+3,v+4)\), it follows that \(\left\lbrace v+2,v+3 \right\rbrace \in Y\). Similarly, it follows that the edges \(\left\lbrace v+3,v+4 \right\rbrace, \left\lbrace v+3,v+6 \right\rbrace\) must belong to \(G\) and \(B\) in some arrangement. Thus, \(\left\lbrace v,v+3 \right\rbrace \in R\). This contradicts with our assumption that \(R\) had a gap at vertices \(v,v+1\). Thus, \(\left\lbrace v,v+3 \right\rbrace\in Y\). A similar argument shows that \(\left\lbrace v+1,v+4 \right\rbrace\in R\).
\end{proof}

A 1-factorisation, \(\mathcal{F}=\left\lbrace F_1,F_2,F_3,F_4 \right\rbrace\), of \(Circ(2n,\left\lbrace 1,3 \right\rbrace)\) that has 4 sequential gaps between \(0\) and \(4\) in the order \(F_1,F_2,F_3,F_4\) that also has the edge \(\left\lbrace 2,3 \right\rbrace\) in \(F_1\) and the edge \(\left\lbrace 1,2 \right\rbrace\) in \(F_4\), is said to satisfy \emph{Condition C}. We now show that we can build a larger 1-factorisation satisfying Condition C from a 1-factorisation that satisfies Condition C.

\begin{lemma} \label{Lem: Inductive Construction}
    If there exists a 1-factorisation of \(Circ(2n,\left\lbrace 1,3 \right\rbrace)\) that satisfies Condition C, then there exists a 1-factorisation of \(Circ(2n+4,\left\lbrace 1,3 \right\rbrace)\) that also satisfies Condition C.
\end{lemma}

\begin{proof}
    Let \(\mathcal{F}\) be a 1-factorisation of \(Circ(2n,\left\lbrace 1,3 \right\rbrace)\) that satisfies Condition C.
    Label the 1-factors of \(\mathcal{F}\) by \(R, B, G, Y\) such that the order of the four gaps is \(R, B, G, Y\). It follows that \(\left\lbrace 2,3 \right\rbrace\in R\) and \(\left\lbrace 1,2 \right\rbrace\in Y\).

    We now use \(\mathcal{F}\) to construct a 1-factorisation, \(\mathcal{F}'=\left\lbrace R',B',G',Y' \right\rbrace\), of \(Circ(2n+4,\left\lbrace 1,3 \right\rbrace)\) as follows:

    \begin{align*}
        R' & = \left\lbrace \left\lbrace v_1+4,v_2+4 \right\rbrace \textrm{ if } \left\lbrace v_1,v_2 \right\rbrace \in E(R) \textrm{ and } 1 \leq v_1,v_2  \right\rbrace \\ &\qquad \cup \left\lbrace \left\lbrace 0,v_1 + 4  \right\rbrace \textrm{if} \left\lbrace 0,v_1 \right\rbrace \in E(R) \right\rbrace \\ & \qquad \cup \left\lbrace \left\lbrace 1,4 \right\rbrace,\left\lbrace 2,3 \right\rbrace \right\rbrace                                                                  \\
        B' & = \left\lbrace \left\lbrace v_1+4,v_2+4 \right\rbrace \textrm{ if } \left\lbrace v_1,v_2 \right\rbrace \in E(B) \textrm{ and } 2 \leq v_1,v_2  \right\rbrace
        \\ &\qquad \cup \left\lbrace
        \left\lbrace u_1,u_2 \right\rbrace \textrm{ if } \left\lbrace u_1,u_2 \right\rbrace \in E(B) \textrm{ and } 0 \leq u_1,u_2 \leq 1
        \right\rbrace
        \\ &\qquad \cup \left\lbrace \left\lbrace u,v_1 + 4  \right\rbrace \textrm{ if } \left\lbrace u_1,v_1 \right\rbrace \in E(B) \textrm{ and } 0 \leq u_1 \leq 1 \textrm{ and } 2 \leq v_1 \right\rbrace                                     \\ & \qquad \cup \left\lbrace \left\lbrace 2,5 \right\rbrace,\left\lbrace 3,4 \right\rbrace \right\rbrace                                                                                                                                                                   \\
        G' & = \left\lbrace \left\lbrace v_1+4,v_2+4 \right\rbrace \textrm{ if } \left\lbrace v_1,v_2 \right\rbrace \in E(G) \textrm{ and } 3 \leq v_1,v_2  \right\rbrace
        \\ &\qquad \cup \left\lbrace
        \left\lbrace u_1,u_2 \right\rbrace \textrm{ if } \left\lbrace u_1,u_2 \right\rbrace \in E(G) \textrm{ and } 0 \leq u_1,u_2 \leq 2
        \right\rbrace
        \\ &\qquad \cup \left\lbrace \left\lbrace u,v_1 + 4  \right\rbrace \textrm{ if } \left\lbrace u_1,v_1 \right\rbrace \in E(G) \textrm{ and } 0 \leq u_1 \leq 2 \textrm{ and } 3 \leq v_1 \right\rbrace                                     \\ & \qquad \cup \left\lbrace \left\lbrace 3,6 \right\rbrace,\left\lbrace 4,5 \right\rbrace \right\rbrace                                                                                                                                                                                                                           \\
        Y' & = \left\lbrace \left\lbrace v_1+4,v_2+4 \right\rbrace \textrm{ if } \left\lbrace v_1,v_2 \right\rbrace \in E(Y) \textrm{ and } 4 \leq v_1,v_2  \right\rbrace
        \\ &\qquad \cup \left\lbrace
        \left\lbrace u_1,u_2 \right\rbrace \textrm{ if } \left\lbrace u_1,u_2 \right\rbrace \in E(Y) \textrm{ and } 0 \leq u_1,u_2 \leq 3
        \right\rbrace
        \\ &\qquad \cup \left\lbrace \left\lbrace u,v_1 + 4  \right\rbrace \textrm{ if } \left\lbrace u_1,v_1 \right\rbrace \in E(Y) \textrm{ and } 0 \leq u_1 \leq 3 \textrm{ and } 4 \leq v_1 \right\rbrace                                     \\ & \qquad \cup \left\lbrace \left\lbrace 4,7 \right\rbrace,\left\lbrace 5,6 \right\rbrace \right\rbrace
    \end{align*}
    It is straightforward to confirm that \(\mathcal{F}'\) is a 1-factorisation of \(Circ(2n+4,\left\lbrace 1,3 \right\rbrace)\) that satisfies Condition C.
\end{proof}

We now consider the cycle structure of pairs of 1-factors of \(\mathcal{F}'\).

\begin{lemma} \label{Lem: Extension Cycle Types}
    Let \(\mathcal{F}\) be a 1-factorisation of \(Circ(2n,\left\lbrace 1,3 \right\rbrace)\) that satisfies Condition C.
    Label the 1-factors of \(\mathcal{F}\) by \(R, B, G, \textrm{and } Y\) such that the order of the four gaps is \(R, B, G, Y\). Further, let \(\mathcal{F}'=\left\lbrace R',B',G',Y' \right\rbrace\) be the 1-factorisation obtained from \(\mathcal{F}\) using the construction from Lemma \ref{Lem: Inductive Construction}.

    Let \(X\) and \(Y\) be 1-factors of \(\mathcal{F}\) that have type \(\left[ c_1, c_2, \dots, c_{t_{X \cup Y}} \right]\) and let \([C_1,C_2, \dots, C_{t_{X\cup Y}} ]\) be the cycles of \(X\cup Y\) where \(|V(C_i)| = c_i\). Then Table \ref{Tab: Pair of Factors types} gives the type of the corresponding 1-factors \(X'\) and \(Y'\) of \(\mathcal{F}'\), depending on which vertex condition is satisfied.

    \begin{table}[H]
        \centering
        \begin{tabular}{|l|l|l|l|}\hline
            \(X\cup Y\) & Type of \(X\cup Y\)                                  & Vertex Condition      & Type of \(X'\cup Y'\)                                    \\ \hline
            \(R\cup B\) & \(\left[ c_1, c_2, \dots, c_{t_{R \cup B}} \right]\) & \(1\in C_1\)          & \(\left[ c_1+4, c_2, \dots, c_{t_{R \cup B}} \right]\)   \\ \hline
            \(B\cup G\) & \(\left[ c_1, c_2, \dots, c_{t_{B \cup G}} \right]\) & \(2\in C_1\)          & \(\left[ c_1+4, c_2, \dots, c_{t_{B \cup G}} \right]\)   \\ \hline
            \(G\cup Y\) & \(\left[ c_1, c_2, \dots, c_{t_{G \cup Y}} \right]\) & \(3\in C_1\)          & \(\left[ c_1+4, c_2, \dots, c_{t_{G \cup Y}} \right]\)   \\ \hline
            \(R\cup G\) & \(\left[ c_1, c_2, \dots, c_{t_{R \cup G}} \right]\) & \(1,2\in C_1\)        & \(\left[ c_1+4, c_2, \dots, c_{t_{R \cup G}} \right]\)   \\ \hline
            \(R\cup G\) & \(\left[ c_1, c_2, \dots, c_{t_{R \cup G}} \right]\) & \(1\in C_1,2\in C_2\) & \(\left[ c_1+2, c_2+2, \dots, c_{t_{R \cup G}} \right]\) \\ \hline
            \(B\cup Y\) & \(\left[ c_1, c_2, \dots, c_{t_{B \cup Y}} \right]\) & \(2,3\in C_1\)        & \(\left[ c_1+4, c_2, \dots, c_{t_{B \cup Y}} \right]\)   \\ \hline
            \(B\cup Y\) & \(\left[ c_1, c_2, \dots, c_{t_{B \cup Y}} \right]\) & \(2\in C_1,3\in C_2\) & \(\left[ c_1+2, c_2+2, \dots, c_{t_{B \cup Y}} \right]\) \\ \hline
            \(R\cup Y\) & \(\left[ c_1, c_2, \dots, c_{t_{R \cup Y}} \right]\) & \(0\in C_1\)          & \(\left[ c_1+4, c_2, \dots, c_{t_{R \cup Y}} \right]\)   \\ \hline
        \end{tabular}
        \caption{The types of pairs of 1-factors of \(\mathcal{F}'\).}\label{Tab: Pair of Factors types}
    \end{table}

\end{lemma}

\begin{proof}
    Let \(V\) and \(V'\) be the vertex sets of \(Circ(2n,\left\lbrace 1,3 \right\rbrace)\) and \(Circ(2n+4,\left\lbrace 1,3 \right\rbrace)\) respectively. For each pair of 1-factors, \(X\) and \(Y\), we define a set of vertices \(S_{X \cup Y}\) and a function \(f_{X \cup Y}: V\setminus S_{X \cup Y} \mapsto V'\) as seen in Table \ref{Tab: Pair of Factors fs}. For any pair of 1-factors, \(X\) and \(Y\), any path \([v_1, v_2, \dots, v_{k} ]\) in \(X \cup Y\) that does not contain any vertices of \(S_{X \cup Y}\) will correspond to the path \([f(v_1), f(v_2), \dots, f(v_{k}) ]\) in \(X'\cup Y'\). Table \ref{Tab: Special Paths}, states what paths in \(X\cup Y\) that contain vertices of \(S_{X \cup Y}\) correspond to in \(X'\cup Y'\). The results follow from these two tables. Also observe that the vertex conditions of \(X\cup Y\) will continue to hold in \(X'\cup Y'\).
    \begin{table}[H]
        \centering
        \begin{tabular}{|l|l|l|}\hline
            \(X\cup Y\) & \(S_{X \cup Y}\)                          & \(f_{X \cup Y}\)                                 \\ \hline
            \(R\cup B\) & \(\left\lbrace 1 \right\rbrace\)          & \(
            f_{R\cup B}(v_i) = \begin{cases}
                                   v_i   & v_i = 0            \\
                                   v_i+4 & 2\leq v_i \leq n-1
                               \end{cases}
            \)                                                                                                         \\ \hline
            \(B\cup G\) & \(\left\lbrace 2 \right\rbrace  \)        & \(
            f_{B\cup G}(v_i) = \begin{cases}
                                   v_i   & 0 \leq v_i \leq 1   \\
                                   v_i+4 & 3 \leq v_i \leq n-1
                               \end{cases}
            \)                                                                                                         \\ \hline
            \(G\cup Y\) & \(\left\lbrace 3 \right\rbrace \)         & \(
            f_{G\cup Y}(v_i) = \begin{cases}
                                   v_i   & 0 \leq v_i \leq 2   \\
                                   v_i+4 & 4 \leq v_i \leq n-1
                               \end{cases}
            \)                                                                                                         \\ \hline
            \(R\cup G\) & \(\left\lbrace 1,2 \right\rbrace \)       & \(f_{R\cup G}(v_i) = \begin{cases}
                                                                                               v_i   & v_i = 0             \\
                                                                                               v_i+4 & 3 \leq v_i \leq n-1
                                                                                           \end{cases}  \) \\ \hline
            \(B\cup Y\) & \(\left\lbrace 2,3 \right\rbrace \)       & \(f_{B\cup Y}(v_i) = \begin{cases}
                                                                                               v_i   & 0 \leq v_i \leq 1   \\
                                                                                               v_i+4 & 4 \leq v_i \leq n-1
                                                                                           \end{cases}\) \\ \hline
            \(R\cup Y\) & \(\left\lbrace 0,1,2,3,4 \right\rbrace \) & \(f_{R\cup Y}(v_i) = \begin{cases}
                                                                                               v_i+4 & 5 \leq v_i \leq n-1
                                                                                           \end{cases}\) \\ \hline
        \end{tabular}
        \caption{Sets and functions for each pair of 1-factors of \(\mathcal{F}\).}\label{Tab: Pair of Factors fs}
    \end{table}
    \begin{table}[h]
        \centering
        \begin{tabular}{|l|l|l|}\hline
            \(X\cup Y\) & Path in \(X\cup Y\)     & Path in \(X'\cup Y'\)                 \\ \hline
            \(R\cup B\) & \([v_1,1,v_2]\)         & \([f(v_1),1,4,3,2,5, f(v_2)]\)        \\ \hline
            \(B\cup G\) & \([v_1,2,v_2]\)         & \([f(v_1),2,5,4,3,6, f(v_2)]\)        \\ \hline
            \(G\cup Y\) & \([v_1,3,v_2]\)         & \([f(v_1),3,6,5,4,7, f(v_2)]\)        \\ \hline
            \(R\cup G\) & \([v_1,1,v_2]\)         & \([f(v_1),1,4,5,f(v_2)]\)             \\ \hline
            \(R\cup G\) & \([u_1,2,u_2]\)         & \([f(u_1),2,3,6,f(u_2)]\)             \\ \hline
            \(B\cup Y\) & \([v_1,2,v_2]\)         & \([f(v_1),2,5,6,f(v_2)]\)             \\ \hline
            \(B\cup Y\) & \([u_1,3,u_2]\)         & \([f(u_1),3,4,7,f(u_2)]\)             \\ \hline
            \(R\cup Y\) & \([v_1,0,3,2,1,4,v_2]\) & \([f(v_1),0,3,2,1,4,7,6,5,8,f(v_2)]\) \\ \hline
        \end{tabular}
        \caption{Paths in \(X\cup Y\), and their corresponding paths in \({X'}\cup{Y'}\)}\label{Tab: Special Paths}
    \end{table}
\end{proof}

We let \(\mathcal{F}^k=\left\lbrace R^k,B^k,G^k,Y^k \right\rbrace\) be the 1-factorisation obtained from \(\mathcal{F}\) using the construction from Lemma \ref{Lem: Inductive Construction} \(k\) consecutive times.

\begin{corollary} \label{Cor: Types after repeated extension}
    Let \(\mathcal{F}\) be a 1-factorisation of \(Circ(2n,\left\lbrace 1,3 \right\rbrace)\) that satisfies Condition C.
    Label the 1-factors of \(\mathcal{F}\) by \(R, B, G, Y\) such that the order of the four gaps is \(R, B, G, Y\). 
    Table \ref{Tab: Pair of Factors types multiple} describes the type of each pair of 1-factors of \(\mathcal{F}^k\).
    \begin{table}[h]
        \centering
        \begin{tabular}{|l|l|l|l|}\hline
            \(X\cup Y\) & Type of \(X\cup Y\)                                  & Vertex Condition      & Type of \(X^k\cup Y^k\)                                    \\ \hline
            \(R\cup B\) & \(\left[ a_1, a_2, \dots, a_{t_{R \cup B}} \right]\) & \(1\in a_1\)          & \(\left[ a_1+4k, a_2, \dots, a_{t_{R \cup B}} \right]\)    \\ \hline
            \(B\cup G\) & \(\left[ a_1, a_2, \dots, a_{t_{B \cup G}} \right]\) & \(2\in a_1\)          & \(\left[ a_1+4k, a_2, \dots, a_{t_{B \cup G}} \right]\)    \\ \hline
            \(G\cup Y\) & \(\left[ a_1, a_2, \dots, a_{t_{G \cup Y}} \right]\) & \(3\in a_1\)          & \(\left[ a_1+4k, a_2, \dots, a_{t_{G \cup Y}} \right]\)    \\ \hline
            \(R\cup G\) & \(\left[ a_1, a_2, \dots, a_{t_{R \cup G}} \right]\) & \(1,2\in a_1\)        & \(\left[ a_1+4k, a_2, \dots, a_{t_{R \cup G}} \right]\)    \\ \hline
            \(R\cup G\) & \(\left[ a_1, a_2, \dots, a_{t_{R \cup G}} \right]\) & \(1\in a_1,2\in a_2\) & \(\left[ a_1+2k, a_2+2k, \dots, a_{t_{R \cup G}} \right]\) \\ \hline
            \(B\cup Y\) & \(\left[ a_1, a_2, \dots, a_{t_{B \cup Y}} \right]\) & \(2,3\in a_1\)        & \(\left[ a_1+4k, a_2, \dots, a_{t_{B \cup Y}} \right]\)    \\ \hline
            \(B\cup Y\) & \(\left[ a_1, a_2, \dots, a_{t_{B \cup Y}} \right]\) & \(2\in a_1,3\in a_2\) & \(\left[ a_1+2k, a_2+2k, \dots, a_{t_{B \cup Y}} \right]\) \\ \hline
            \(R\cup Y\) & \(\left[ a_1, a_2, \dots, a_{t_{R \cup Y}} \right]\) & \(0\in a_1\)          & \(\left[ a_1+4k, a_2, \dots, a_{t_{R \cup Y}} \right]\)    \\ \hline
        \end{tabular}
        \caption{The types of pairs of 1-factors of \(\mathcal{F}^k\).}\label{Tab: Pair of Factors types multiple}
    \end{table}
\end{corollary}

With the above we can now prove some existence results of \(m\)-B1Fs of \(Circ(2n,\left\lbrace 1,3 \right\rbrace)\).

\begin{theorem}\label{Thm: 2-B1Fs of Circ(n,{1,3})}
    The 4-regular circulant graph \(Circ(2n,\left\lbrace 1,3 \right\rbrace)\) admits a 2-B1F if and only if \(n \geq 5 \).
\end{theorem}

\begin{proof}
    For \(n \leq 4\), we enumerated all possible 1-factorisations of \(Circ(2n,\left\lbrace 1,3 \right\rbrace)\) by computer and determined that none are 2-B1Fs. It remains to show that there exists a 2-B1F of \(Circ(2n, \left\lbrace 1,3 \right\rbrace)\) for \(n \geq 5 \). For \(n=5\) and \(n=6\) we give these explicitly.

    The following 1-factorisation of \(Circ(10,\left\lbrace 1,3 \right\rbrace)\) with factors \( R,B,G,Y\) is a 2-B1F with types  \((\left[10\right],\left[6,4\right])\).
    \begin{align*}
        R & = \left\lbrace \left\lbrace 0,1 \right\rbrace, \left\lbrace 2,3 \right\rbrace,\left\lbrace 4,5 \right\rbrace,\left\lbrace 6,7 \right\rbrace,\left\lbrace 8,9 \right\rbrace \right\rbrace \\
        B & = \left\lbrace \left\lbrace 0,3 \right\rbrace,\left\lbrace 1,2 \right\rbrace,\left\lbrace 4,7 \right\rbrace,\left\lbrace 5,8 \right\rbrace,\left\lbrace 6,9 \right\rbrace \right\rbrace  \\
        G & = \left\lbrace \left\lbrace 0,9 \right\rbrace,\left\lbrace 1,4 \right\rbrace,\left\lbrace 2,5 \right\rbrace,\left\lbrace 3,6 \right\rbrace,\left\lbrace 7,8 \right\rbrace \right\rbrace  \\
        Y & = \left\lbrace \left\lbrace 0,7 \right\rbrace,\left\lbrace 1,8 \right\rbrace,\left\lbrace 2,9 \right\rbrace,\left\lbrace 3,4 \right\rbrace,\left\lbrace 5,6 \right\rbrace \right\rbrace
    \end{align*}

    The following 1-factorisation of \(Circ(12,\left\lbrace 1,3 \right\rbrace)\) with factors \( R,B,G,Y\) is a 2-B1F with types  \((\left[12\right],\left[8,4\right])\).
    \begin{align*}
        R & = \left\lbrace
        \left\lbrace
        0,1
        \right\rbrace,
        \left\lbrace
        2,5
        \right\rbrace,
        \left\lbrace
        3,4
        \right\rbrace,
        \left\lbrace
        6,9
        \right\rbrace,
        \left\lbrace
        7,10
        \right\rbrace,
        \left\lbrace
        8,11
        \right\rbrace
        \right\rbrace
        \\
        B & = \left\lbrace
        \left\lbrace
        0,3
        \right\rbrace,
        \left\lbrace
        1,4
        \right\rbrace,
        \left\lbrace
        2,11
        \right\rbrace,
        \left\lbrace
        5,6
        \right\rbrace,
        \left\lbrace
        7,8
        \right\rbrace,
        \left\lbrace
        9,10
        \right\rbrace
        \right\rbrace
        \\
        G & = \left\lbrace
        \left\lbrace
        0,11
        \right\rbrace,
        \left\lbrace
        1,10
        \right\rbrace,
        \left\lbrace
        2,3
        \right\rbrace,
        \left\lbrace
        4,5
        \right\rbrace,
        \left\lbrace
        6,7
        \right\rbrace,
        \left\lbrace
        8,9
        \right\rbrace
        \right\rbrace
        \\
        Y & = \left\lbrace
        \left\lbrace
        0,9
        \right\rbrace,
        \left\lbrace
        1,2
        \right\rbrace,
        \left\lbrace
        3,6
        \right\rbrace,
        \left\lbrace
        4,7
        \right\rbrace,
        \left\lbrace
        5,8
        \right\rbrace,
        \left\lbrace
        10,11
        \right\rbrace
        \right\rbrace
    \end{align*}

    For \(n \geq 7 \) we will show the existence of 2-B1Fs of \(Circ(16+4k,\left\lbrace 1,3 \right\rbrace)\) and \(Circ(14+4k,\left\lbrace 1,3 \right\rbrace)\) for all \(k \geq 0 \) seperately.

    The following 1-factorisation, \(\mathcal{F}_\gamma=\left\lbrace R,B,G,Y \right\rbrace\), of \(Circ(16,\left\lbrace 1,3 \right\rbrace)\) is a 2-B1F with types \(\left(\left[16\right],\left[12,4\right]\right)\).
    \begin{align*}
        R & = \left\lbrace\left\lbrace0, 15\right\rbrace, \left\lbrace1, 4\right\rbrace, \left\lbrace2, 3\right\rbrace, \left\lbrace5, 8\right\rbrace, \left\lbrace6, 9\right\rbrace, \left\lbrace7, 10\right\rbrace, \left\lbrace11, 14\right\rbrace, \left\lbrace12, 13\right\rbrace\right\rbrace \\
        B & = \left\lbrace\left\lbrace0, 13\right\rbrace, \left\lbrace1, 14\right\rbrace, \left\lbrace2, 5\right\rbrace, \left\lbrace3, 4\right\rbrace, \left\lbrace6, 7\right\rbrace, \left\lbrace8, 9\right\rbrace, \left\lbrace10, 11\right\rbrace, \left\lbrace12, 15\right\rbrace\right\rbrace \\
        G & = \left\lbrace\left\lbrace0, 1\right\rbrace, \left\lbrace2, 15\right\rbrace, \left\lbrace3, 6\right\rbrace,\left\lbrace4, 5\right\rbrace, \left\lbrace7, 8\right\rbrace, \left\lbrace9, 10\right\rbrace, \left\lbrace11, 12\right\rbrace, \left\lbrace13, 14\right\rbrace\right\rbrace  \\
        Y & = \left\lbrace\left\lbrace0, 3\right\rbrace, \left\lbrace1, 2\right\rbrace, \left\lbrace4, 7\right\rbrace, \left\lbrace5, 6\right\rbrace, \left\lbrace8, 11\right\rbrace, \left\lbrace9, 12\right\rbrace, \left\lbrace10, 13\right\rbrace, \left\lbrace14, 15\right\rbrace\right\rbrace
    \end{align*}
    It is straightforward to check that \(\mathcal{F}_\gamma\) satisfies Condition C. It is also straightforward to check that
    \begin{multicols}{2}
    \begin{itemize}
        \item \(1\) lies on a \(12\)-cycle in \(R\cup B\),
        \item \(2\) lies on a \(12\)-cycle in \(B\cup G\),
        \item \(3\) lies on a \(16\)-cycle in \(G\cup Y\),
    \end{itemize}
    \columnbreak
    \begin{itemize}
        \item \(1\) and \(2\) lie on a \(12\)-cycle in \(R\cup G\),
        \item \(2\) and \(3\) lie on a \(16\)-cycle in \(B\cup Y\),
        \item \(0\) lies on a \(16\)-cycle in \(R\cup Y\).
    \end{itemize}
    \end{multicols}
    Thus, by Corollary \ref{Cor: Types after repeated extension}, \(\mathcal{F}_\gamma^k\) is a 1-factorisation of \(Circ(16+4k,\left\lbrace 1,3 \right\rbrace)\) that will have three pairs of 1-factors with type \([16+4k]\) and three pairs of 1-factors with type \([12+4k,4]\). As these types are clearly distinct for every integer \(k\), \(\mathcal{F}_\gamma^k\) is a 2-B1F of \(Circ(16+4k,\left\lbrace 1,3 \right\rbrace)\) with types \(([16+4k],[12+4k,4])\) for all \(k \geq 0 \).

    Now, we claim that the following 1-factorisation, \(\mathcal{F}_\delta=\left\lbrace R,B,G,Y \right\rbrace\), of \(Circ(14,\left\lbrace 1,3 \right\rbrace)\) is a 2-B1F with types \(\left(\left[14\right],\left[10,4\right]\right)\).
    \begin{align*}
        R & = \left\lbrace\left\lbrace0, 13\right\rbrace, \left\lbrace1, 4\right\rbrace, \left\lbrace2, 3\right\rbrace, \left\lbrace5, 8\right\rbrace, \left\lbrace6, 7\right\rbrace, \left\lbrace9, 12\right\rbrace, \left\lbrace10, 11\right\rbrace\right\rbrace \\
        B & = \left\lbrace\left\lbrace0, 1\right\rbrace, \left\lbrace2, 5\right\rbrace, \left\lbrace3, 6\right\rbrace, \left\lbrace4, 7\right\rbrace, \left\lbrace8, 9\right\rbrace, \left\lbrace10, 13\right\rbrace, \left\lbrace11, 12\right\rbrace\right\rbrace \\
        G & = \left\lbrace\left\lbrace0, 11\right\rbrace, \left\lbrace1, 12\right\rbrace, \left\lbrace2, 13\right\rbrace, \left\lbrace3, 4\right\rbrace, \left\lbrace5, 6\right\rbrace, \left\lbrace7, 8\right\rbrace, \left\lbrace9, 10\right\rbrace\right\rbrace \\
        Y & = \left\lbrace\left\lbrace0, 3\right\rbrace, \left\lbrace1, 2\right\rbrace, \left\lbrace4, 5\right\rbrace, \left\lbrace6, 9\right\rbrace, \left\lbrace7, 10\right\rbrace, \left\lbrace8, 11\right\rbrace, \left\lbrace12, 13\right\rbrace\right\rbrace
    \end{align*}
    It is straightforward to check that \(\mathcal{F}_\delta\) satisfies Condition C. It is also clear that
    \begin{multicols}{2}
    \begin{itemize}
        \item \(1\) lies on a \(14\)-cycle in \(R\cup B\),
        \item \(2\) lies on a \(10\)-cycle in \(B\cup G\),
        \item \(3\) lies on a \(10\)-cycle in \(G\cup Y\),
    \end{itemize}
    \columnbreak
    \begin{itemize}
        \item \(1\) and \(2\) lie on a \(10\)-cycle in \(R\cup G\),
        \item \(2\) and \(3\) lie on a \(14\)-cycle in \(B\cup Y\),
        \item and \(0\) lies on a \(14\)-cycle in \(R\cup Y\).
    \end{itemize}
\end{multicols}
    Thus, by Corollary \ref{Cor: Types after repeated extension}, \(\mathcal{F}_\delta^k\) is a 1-factorisation of \(Circ(14+4k,\left\lbrace 1,3 \right\rbrace)\) that will have three pairs of 1-factors with type \([14+4k]\) and three pairs of 1-factors with type \([10+4k,4]\). As these types are clearly distinct for every integer \(k\), \(\mathcal{F}_\delta^k\) is a 2-B1F of \(Circ(14+4k,\left\lbrace 1,3 \right\rbrace)\) with types \(([14+4k],[10+4k,4])\) for all \(k \geq 0 \).

    Hence, we have shown that \(Circ(2n,\left\lbrace 1,3 \right\rbrace)\) admits a 2-B1F if and only if \(n \geq 5 \).
\end{proof}

\begin{theorem}\label{Thm: 3-B1Fs of Circ(n,{1,3})}
    The 4-regular circulant graph \(Circ(2n,\left\lbrace 1,3 \right\rbrace)\) admits a 3-B1F if and only if \(n \geq 6 \).
\end{theorem}

\begin{proof}
    For even \(n \leq 5\), we enumerated all possible 1-factorisations of \(Circ(2n,\left\lbrace 1,3 \right\rbrace)\) by computer and determined that none are 3-B1Fs. Thus it remains to show that there exists a 3-B1F of \(Circ(2n, \left\lbrace 1,3 \right\rbrace)\) for all \(n \geq 6 \). For \(n=6,7, \text{ and } 8\) we give these explicitly.

    When \(n=6\) we define the 1-factorisation of \(Circ(12,\left\lbrace 1,3 \right\rbrace)\) \(\mathcal{F}_\alpha=\left\lbrace F_1,F_2,F_3,F_4 \right\rbrace\) as follows:
    \begin{align*}
        F_a & = \left\lbrace \left\lbrace 0,1 \right\rbrace, \left\lbrace 2,3 \right\rbrace,\left\lbrace 4,5 \right\rbrace,\left\lbrace 6,7 \right\rbrace,\left\lbrace 8,9 \right\rbrace,\left\lbrace 10,11 \right\rbrace \right\rbrace, \\
        F_b & = \left\lbrace \left\lbrace 0,3 \right\rbrace,\left\lbrace 1,2 \right\rbrace,\left\lbrace 4,7 \right\rbrace,\left\lbrace 5,6 \right\rbrace,\left\lbrace 8,11 \right\rbrace,\left\lbrace 9,10 \right\rbrace \right\rbrace,  \\
        F_c & = \left\lbrace \left\lbrace 0,11 \right\rbrace,\left\lbrace 1,10 \right\rbrace,\left\lbrace 2,5 \right\rbrace,\left\lbrace 3,4 \right\rbrace,\left\lbrace 6,9 \right\rbrace, \left\lbrace 7,8 \right\rbrace \right\rbrace, \\
        F_d & = \left\lbrace \left\lbrace 0,9 \right\rbrace,\left\lbrace 1,4 \right\rbrace,\left\lbrace 2,11 \right\rbrace,\left\lbrace 3,6 \right\rbrace,\left\lbrace 5,8 \right\rbrace,\left\lbrace 7,10 \right\rbrace \right\rbrace.
    \end{align*}
    It is straightforward to confirm that \(\mathcal{F}_\alpha\) is a 3-B1F with types \(\left(\left[12\right],\left[6^2\right],\left[4^3\right]\right)\).

    When \(n=7\) we define the 1-factorisation of \(Circ(14,\{1,3\})\) \(\mathcal{F}_\beta=\left\lbrace F_1,F_2,F_3,F_4 \right\rbrace\) as follows:
    \begin{align*}
        F_1 & = \left\lbrace
        \left\lbrace
        0,1
        \right\rbrace,
        \left\lbrace
        2,3
        \right\rbrace,
        \left\lbrace
        4,5
        \right\rbrace,
        \left\lbrace
        6,7
        \right\rbrace,
        \left\lbrace
        8,9
        \right\rbrace,
        \left\lbrace
        10,13
        \right\rbrace,
        \left\lbrace
        11,12
        \right\rbrace
        \right\rbrace        \\
        F_2 & = \left\lbrace
        \left\lbrace
        0,13
        \right\rbrace,
        \left\lbrace
        1,2
        \right\rbrace,
        \left\lbrace
        3,4
        \right\rbrace,
        \left\lbrace
        5,6
        \right\rbrace,
        \left\lbrace
        7,8
        \right\rbrace,
        \left\lbrace
        9,12
        \right\rbrace,
        \left\lbrace
        10,11
        \right\rbrace
        \right\rbrace        \\
        F_3 & = \left\lbrace
        \left\lbrace
        0,3
        \right\rbrace,
        \left\lbrace
        1,4
        \right\rbrace,
        \left\lbrace
        2,5
        \right\rbrace,
        \left\lbrace
        6,9
        \right\rbrace,
        \left\lbrace
        7,10
        \right\rbrace,
        \left\lbrace
        8,11
        \right\rbrace,
        \left\lbrace
        12,13
        \right\rbrace
        \right\rbrace        \\
        F_4 & = \left\lbrace
        \left\lbrace
        0,11
        \right\rbrace,
        \left\lbrace
        1,12
        \right\rbrace,
        \left\lbrace
        2,13
        \right\rbrace,
        \left\lbrace
        3,6
        \right\rbrace,
        \left\lbrace
        4,7
        \right\rbrace,
        \left\lbrace
        5,8
        \right\rbrace,
        \left\lbrace
        9,10
        \right\rbrace
        \right\rbrace
    \end{align*}

    It is straightforward to confirm that \(\mathcal{F}_\beta\) is a 3-B1F with types \((\left[14\right],\left[10,4\right],\left[8,6\right])\).

    When \(n=8\) we define the 1-factorisation of \(Circ(16,\{1,3\})\) \(\mathcal{F}_\gamma=\left\lbrace F_1,F_2,F_3,F_4 \right\rbrace\) as follows:
    \begin{align*}
        F_a & = \left\lbrace
        \left\lbrace
        0,1
        \right\rbrace,
        \left\lbrace
        2,3
        \right\rbrace,
        \left\lbrace
        4,5
        \right\rbrace,
        \left\lbrace
        6,9
        \right\rbrace,
        \left\lbrace
        7,8
        \right\rbrace,
        \left\lbrace
        10,13
        \right\rbrace,
        \left\lbrace
        11,14
        \right\rbrace,
        \left\lbrace
        12,15
        \right\rbrace
        \right\rbrace        \\
        F_b & = \left\lbrace
        \left\lbrace
        0,3
        \right\rbrace,
        \left\lbrace
        1,4
        \right\rbrace,
        \left\lbrace
        2,15
        \right\rbrace,
        \left\lbrace
        5,6
        \right\rbrace,
        \left\lbrace
        7,10
        \right\rbrace,
        \left\lbrace
        8,11
        \right\rbrace,
        \left\lbrace
        9,12
        \right\rbrace,
        \left\lbrace
        13,14
        \right\rbrace
        \right\rbrace        \\
        F_c & = \left\lbrace
        \left\lbrace
        0,15
        \right\rbrace,
        \left\lbrace
        1,14
        \right\rbrace,
        \left\lbrace
        2,5
        \right\rbrace,
        \left\lbrace
        3,6
        \right\rbrace,
        \left\lbrace
        4,7
        \right\rbrace,
        \left\lbrace
        8,9
        \right\rbrace,
        \left\lbrace
        10,11
        \right\rbrace,
        \left\lbrace
        12,13
        \right\rbrace
        \right\rbrace        \\
        F_d & = \left\lbrace
        \left\lbrace
        0,13
        \right\rbrace,
        \left\lbrace
        1,2
        \right\rbrace,
        \left\lbrace
        3,4
        \right\rbrace,
        \left\lbrace
        5,8
        \right\rbrace,
        \left\lbrace
        6,7
        \right\rbrace,
        \left\lbrace
        9,10
        \right\rbrace,
        \left\lbrace
        11,12
        \right\rbrace,
        \left\lbrace
        14,15
        \right\rbrace
        \right\rbrace
    \end{align*}

    It is straightforward to confirm that \(\mathcal{F}_\gamma\) is a 3-B1F with types \((\left[12,4\right],\left[10,6\right],\left[8,8\right])\).

    For \(n \geq 9 \) we will show the existence of 3-B1Fs of \(Circ(18+4k,\left\lbrace 1,3 \right\rbrace)\) and \(Circ(20+4k,\left\lbrace 1,3 \right\rbrace)\) for all \(k \geq 0 \) seperately.

    We claim that the following 1-factorisation, \(\mathcal{F}_\gamma=\left\lbrace R,B,G,Y \right\rbrace\), of \(Circ(18,\left\lbrace 1,3 \right\rbrace)\) is a 3-B1F with types \(\left(\left[18\right],\left[14,4\right],\left[12,6\right]\right)\).
    \begin{align*}
        R & = \left\lbrace\left\lbrace0, 17\right\rbrace, \left\lbrace1, 4\right\rbrace, \left\lbrace2, 3\right\rbrace, \left\lbrace5, 6\right\rbrace, \left\lbrace7, 8\right\rbrace, \left\lbrace9, 10\right\rbrace, \left\lbrace11, 12\right\rbrace, \left\lbrace13, 16\right\rbrace, \left\lbrace18, 19\right\rbrace\right\rbrace \\
        B & = \left\lbrace\left\lbrace0, 19\right\rbrace, \left\lbrace1, 18\right\rbrace, \left\lbrace2, 5\right\rbrace, \left\lbrace3, 4\right\rbrace, \left\lbrace6, 7\right\rbrace, \left\lbrace8, 9\right\rbrace, \left\lbrace10, 11\right\rbrace, \left\lbrace12, 13\right\rbrace, \left\lbrace16, 17\right\rbrace\right\rbrace \\
        G & = \left\lbrace\left\lbrace0, 1\right\rbrace, \left\lbrace2, 19\right\rbrace, \left\lbrace3, 6\right\rbrace, \left\lbrace4, 7\right\rbrace, \left\lbrace5, 8\right\rbrace, \left\lbrace9, 12\right\rbrace, \left\lbrace10, 13\right\rbrace, \left\lbrace11, 16\right\rbrace, \left\lbrace17, 18\right\rbrace\right\rbrace \\
        Y & = \left\lbrace\left\lbrace0, 3\right\rbrace, \left\lbrace1, 2\right\rbrace, \left\lbrace4, 5\right\rbrace, \left\lbrace6, 9\right\rbrace, \left\lbrace7, 10\right\rbrace, \left\lbrace8, 11\right\rbrace, \left\lbrace12, 17\right\rbrace, \left\lbrace13, 18\right\rbrace, \left\lbrace16, 19\right\rbrace\right\rbrace
    \end{align*}
    It is straightforward to check that \(\mathcal{F}_\gamma\) satisfies Condition C. It is also straightforward to check that
    \begin{multicols}{2}
    \begin{itemize}
        \item \(1\) lies on a \(18\)-cycle in \(R\cup B\),
        \item \(2\) lies on a \(14\)-cycle in \(B\cup G\),
        \item \(3\) lies on a \(18\)-cycle in \(G\cup Y\),
    \end{itemize}
    \columnbreak
    \begin{itemize}
        \item \(1\) and \(2\) lie on a \(12\)-cycle in \(R\cup G\),
        \item \(2\) and \(3\) lie on a \(12\)-cycle in \(B\cup Y\),
        \item \(0\) and lies on a \(14\)-cycle in \(R\cup Y\).
    \end{itemize}
\end{multicols}
    Thus, by Corollary \ref{Cor: Types after repeated extension}, \(\mathcal{F}_\gamma^k\) is a 1-factorisation of \(Circ(18+4k,\left\lbrace 1,3 \right\rbrace)\) that will have two pairs of 1-factors with type \([18+4k]\), two pairs of 1-factors with type \([14+4k,4]\), and two pairs of 1-factors with type \([12+4k,6]\). As these types are clearly distinct for every integer \(k\), \(\mathcal{F}_\gamma^k\) is a 3-B1F of \(Circ(18+4k,\left\lbrace 1,3 \right\rbrace)\) with types \(\left(\left[18+4k\right],\left[14+4k,4\right],\left[12+4k,6\right]\right)\) for all \(k \geq 0 \).

    Now, we claim that the following 1-factorisation, \(\mathcal{F}_\delta=\left\lbrace R,B,G,Y \right\rbrace\), of \(Circ(20,\left\lbrace 1,3 \right\rbrace)\) is a 3-B1F with types \(\left(\left[16,4\right],\left[14,6\right],\left[12,8\right]\right)\).
    \begin{align*}
        R & = \left\lbrace\left\lbrace0, 19\right\rbrace, \left\lbrace1, 4\right\rbrace, \left\lbrace2, 3\right\rbrace, \left\lbrace5, 8\right\rbrace, \left\lbrace6, 9\right\rbrace, \left\lbrace7, 10\right\rbrace, \left\lbrace11, 14\right\rbrace, \left\lbrace12, 13\right\rbrace, \left\lbrace15, 16\right\rbrace, \left\lbrace17, 18\right\rbrace\right\rbrace \\
        B & = \left\lbrace\left\lbrace0, 17\right\rbrace, \left\lbrace1, 18\right\rbrace, \left\lbrace2, 5\right\rbrace, \left\lbrace3, 4\right\rbrace, \left\lbrace6, 7\right\rbrace, \left\lbrace8, 11\right\rbrace, \left\lbrace9, 12\right\rbrace, \left\lbrace10, 13\right\rbrace, \left\lbrace14, 15\right\rbrace, \left\lbrace16, 19\right\rbrace\right\rbrace \\
        G & = \left\lbrace\left\lbrace0, 1\right\rbrace, \left\lbrace2, 19\right\rbrace, \left\lbrace3, 6\right\rbrace, \left\lbrace4, 5\right\rbrace, \left\lbrace7, 8\right\rbrace, \left\lbrace9, 10\right\rbrace, \left\lbrace11, 12\right\rbrace, \left\lbrace13, 16\right\rbrace, \left\lbrace14, 17\right\rbrace, \left\lbrace15, 18\right\rbrace\right\rbrace \\
        Y & = \left\lbrace\left\lbrace0, 3\right\rbrace, \left\lbrace1, 2\right\rbrace, \left\lbrace4, 7\right\rbrace, \left\lbrace5, 6\right\rbrace, \left\lbrace8, 9\right\rbrace, \left\lbrace10, 11\right\rbrace, \left\lbrace12, 15\right\rbrace, \left\lbrace13, 14\right\rbrace, \left\lbrace16, 17\right\rbrace, \left\lbrace18, 19\right\rbrace\right\rbrace
    \end{align*}
    It is straightforward to check that \(\mathcal{F}_\delta\) satisfies Condition C. It is also straightforward to check that
    \begin{multicols}{2}
    \begin{itemize}
        \item \(1\) lies on a \(14\)-cycle in \(R\cup B\),
        \item \(2\) lies on a \(14\)-cycle in \(B\cup G\),
        \item \(3\) lies on a \(16\)-cycle in \(G\cup Y\),
    \end{itemize}
    \columnbreak
    \begin{itemize}
        \item \(1\) and \(2\) lie on a \(12\)-cycle in \(R\cup G\),
        \item \(2\) and \(3\) lie on a \(12\)-cycle in \(B\cup Y\),
        \item \(0\) lies on a \(16\)-cycle in \(R\cup Y\),
    \end{itemize}
\end{multicols}
    Thus, by Corollary \ref{Cor: Types after repeated extension}, \(\mathcal{F}_\delta^k\) is a 1-factorisation of \(Circ(20+4k,\left\lbrace 1,3 \right\rbrace)\) that will have two pairs of 1-factors with type \([16+4k,4]\), two pairs of 1-factors with type \([14+4k,6]\), and two pairs of 1-factors with type \([12+4k,8]\). As these types are clearly distinct for every integer \(k\), \(\mathcal{F}_\gamma^k\) is a 3-B1F of \(Circ(20+4k,\left\lbrace 1,3 \right\rbrace)\) with types \(\left(\left[16+4k,4\right],\left[14+4k,6\right],\left[12+4k,8\right]\right)\) for all \(k \geq 0 \).

    Thus, we have shown that \(Circ(2n,\left\lbrace 1,3 \right\rbrace)\) admits a 3-B1F if and only if \(n \geq 6 \).
\end{proof}

\begin{theorem}\label{Thm: 6-B1Fs of Circ(n,{1,3})}
    The 4-regular circulant graph \(Circ(2n,\left\lbrace 1,3 \right\rbrace)\) admits a 6-B1F if and only if \(n \geq 9 \).
\end{theorem}

\begin{proof}
    For \(n \leq 8\), we enumerated all possible 1-factorisations of \(Circ(2n,\left\lbrace 1,3 \right\rbrace)\) by computer and determined that none are 6-B1Fs. Thus it remains to show that there exists a 6-B1F of \(Circ(2n, \left\lbrace 1,3 \right\rbrace)\) for all \(n \geq 9\).

    We claim that the following 1-factorisation, \(\mathcal{F}_\alpha=\left\lbrace R,G,B,Y \right\rbrace\) is a 6-B1F with types \((\left[18\right],\left[14,4\right],\left[12,6\right],\left[10,8\right],\left[10,4,4\right],\left[8,6,4\right])\).
    \begin{align*}
        R & = \left\lbrace
        \left\lbrace
        0,1
        \right\rbrace,
        \left\lbrace
        2,3
        \right\rbrace,
        \left\lbrace
        4,5
        \right\rbrace,
        \left\lbrace
        6,9
        \right\rbrace,
        \left\lbrace
        7,8
        \right\rbrace,
        \left\lbrace
        10,11
        \right\rbrace,
        \left\lbrace
        12,13
        \right\rbrace,
        \left\lbrace
        14,17
        \right\rbrace,
        \left\lbrace
        15,16
        \right\rbrace
        \right\rbrace        \\
        B & = \left\lbrace
        \left\lbrace
        0,17
        \right\rbrace,
        \left\lbrace
        1,2
        \right\rbrace,
        \left\lbrace
        3,4
        \right\rbrace,
        \left\lbrace
        5,6
        \right\rbrace,
        \left\lbrace
        7,10
        \right\rbrace,
        \left\lbrace
        8,9
        \right\rbrace,
        \left\lbrace
        11,14
        \right\rbrace,
        \left\lbrace
        12,15
        \right\rbrace,
        \left\lbrace
        13,16
        \right\rbrace
        \right\rbrace        \\
        G & = \left\lbrace
        \left\lbrace
        0,3
        \right\rbrace,
        \left\lbrace
        1,4
        \right\rbrace,
        \left\lbrace
        2,5
        \right\rbrace,
        \left\lbrace
        6,7
        \right\rbrace,
        \left\lbrace
        8,11
        \right\rbrace,
        \left\lbrace
        9,12
        \right\rbrace,
        \left\lbrace
        10,13
        \right\rbrace,
        \left\lbrace
        14,15
        \right\rbrace,
        \left\lbrace
        16,17
        \right\rbrace
        \right\rbrace        \\
        Y & = \left\lbrace
        \left\lbrace
        0,15
        \right\rbrace,
        \left\lbrace
        1,16
        \right\rbrace,
        \left\lbrace
        2,17
        \right\rbrace,
        \left\lbrace
        3,6
        \right\rbrace,
        \left\lbrace
        4,7
        \right\rbrace,
        \left\lbrace
        5,8
        \right\rbrace,
        \left\lbrace
        9,10
        \right\rbrace,
        \left\lbrace
        11,12
        \right\rbrace,
        \left\lbrace
        13,14
        \right\rbrace
        \right\rbrace
    \end{align*}

    For \(n \geq 10 \) we will show the existence of 6-B1Fs of \(Circ(20+4k,\left\lbrace 1,3 \right\rbrace)\) and \(Circ(22+4k,\left\lbrace 1,3 \right\rbrace)\) for all \(k \geq 0 \) seperately.
    We claim that the following 1-factorisation, \({\mathcal{F}_\beta=\left\lbrace R,B,G,Y \right\rbrace}\), of \(Circ(20,\left\lbrace 1,3 \right\rbrace)\) is a 6-B1F with types \[\left(\left[20\right],\left[16,4\right],\left[14,6\right],\left[12,8\right],\left[12,4,4\right],\left[10,6,4\right]\right).\]
    \begin{align*}
        R & = \left\lbrace\left\lbrace0, 19\right\rbrace, \left\lbrace1, 4\right\rbrace, \left\lbrace2, 3\right\rbrace, \left\lbrace5, 8\right\rbrace, \left\lbrace6, 7\right\rbrace, \left\lbrace9, 10\right\rbrace, \left\lbrace11, 12\right\rbrace, \left\lbrace13, 16\right\rbrace, \left\lbrace14, 17\right\rbrace, \left\lbrace15, 18\right\rbrace\right\rbrace \\
        B & = \left\lbrace\left\lbrace0, 1\right\rbrace, \left\lbrace2, 5\right\rbrace, \left\lbrace3, 6\right\rbrace, \left\lbrace4, 7\right\rbrace, \left\lbrace8, 9\right\rbrace, \left\lbrace10, 13\right\rbrace, \left\lbrace11, 14\right\rbrace, \left\lbrace12, 15\right\rbrace, \left\lbrace16, 19\right\rbrace, \left\lbrace17, 18\right\rbrace\right\rbrace \\
        G & = \left\lbrace\left\lbrace0, 17\right\rbrace, \left\lbrace1, 18\right\rbrace, \left\lbrace2, 19\right\rbrace, \left\lbrace3, 4\right\rbrace, \left\lbrace5, 6\right\rbrace, \left\lbrace7, 10\right\rbrace, \left\lbrace8, 11\right\rbrace, \left\lbrace9, 12\right\rbrace, \left\lbrace13, 14\right\rbrace, \left\lbrace15, 16\right\rbrace\right\rbrace \\
        Y & = \left\lbrace\left\lbrace0, 3\right\rbrace, \left\lbrace1, 2\right\rbrace, \left\lbrace4, 5\right\rbrace, \left\lbrace6, 9\right\rbrace, \left\lbrace7, 8\right\rbrace, \left\lbrace10, 11\right\rbrace, \left\lbrace12, 13\right\rbrace, \left\lbrace14, 15\right\rbrace, \left\lbrace16, 17\right\rbrace, \left\lbrace18, 19\right\rbrace\right\rbrace
    \end{align*}
    It is straightforward to check that \(\mathcal{F}_\beta\) satisfies Condition C. It is also straightforward to check that
    \begin{multicols}{2}
    \begin{itemize}
        \item \(1\) lies on a \(14\)-cycle in \(R\cup B\),
        \item \(2\) lies on a \(16\)-cycle in \(B\cup G\),
        \item \(3\) lies on a \(12\)-cycle in \(G\cup Y\),
    \end{itemize}
    \columnbreak
    \begin{itemize}
        \item \(1\) and \(2\) lie on a \(12\)-cycle in \(R\cup G\),
        \item \(2\) and \(3\) lie on a \(10\)-cycle in \(B\cup Y\),
        \item and \(0\) lies on a \(20\)-cycle in \(R\cup Y\).
    \end{itemize}
\end{multicols}
    Thus, by Corollary \ref{Cor: Types after repeated extension}, \(\mathcal{F}_\beta^k\) is a 1-factorisation of \(Circ(20+4k,\left\lbrace 1,3 \right\rbrace)\) with the following types:
    \begin{multicols}{2}
        \begin{itemize}
        \item \([14+4k,6]\),
        \item \([16+4k,4]\),
        \item \([12+4k,4,4]\),
    \end{itemize}
    \columnbreak
    \begin{itemize}
        \item \([12+4k,8]\),
        \item \([10+4k,6,4]\),
        \item and \([20+4k]\).
    \end{itemize}
\end{multicols}
    As these types are clearly distinct for every integer \(k\), \(\mathcal{F}_\beta^k\) is a 6-B1F of \(Circ(20+4k,\left\lbrace 1,3 \right\rbrace)\) with types \[\left(\left[20+4k\right],\left[16+4k,4\right],\left[14+4k,6\right],\left[12+4k,8\right],\left[12+4k,4,4\right],\left[10+4k,6,4\right]\right)\] for all \(k \geq 0 \).

    Now, we claim that the following 1-factorisation, \(\mathcal{F}_\gamma=\left\lbrace R,B,G,Y \right\rbrace\), of \(Circ(22,\left\lbrace 1,3 \right\rbrace)\) is a 6-B1F with types \(\left(\left[22\right],\left[18,4\right],\left[16,6\right],\left[14,8\right],\left[14,4,4\right],\left[8,8,4\right]\right).\)
    \begin{align*}
        R & = \left\lbrace\left\lbrace0, 19\right\rbrace, \left\lbrace1, 4\right\rbrace, \left\lbrace2, 3\right\rbrace, \left\lbrace5, 6\right\rbrace, \left\lbrace7, 8\right\rbrace, \left\lbrace9, 12\right\rbrace, \left\lbrace10, 11\right\rbrace, \left\lbrace13, 14\right\rbrace, \left\lbrace15, 16\right\rbrace, \left\lbrace17, 18\right\rbrace, \left\lbrace20, 21\right\rbrace\right\rbrace \\
        B & = \left\lbrace\left\lbrace0, 21\right\rbrace, \left\lbrace1, 20\right\rbrace, \left\lbrace2, 5\right\rbrace, \left\lbrace3, 6\right\rbrace, \left\lbrace4, 7\right\rbrace, \left\lbrace8, 11\right\rbrace, \left\lbrace9, 10\right\rbrace, \left\lbrace12, 13\right\rbrace, \left\lbrace14, 15\right\rbrace, \left\lbrace16, 17\right\rbrace, \left\lbrace18, 19\right\rbrace\right\rbrace \\
        G & = \left\lbrace\left\lbrace0, 1\right\rbrace, \left\lbrace2, 21\right\rbrace, \left\lbrace3, 4\right\rbrace, \left\lbrace5, 8\right\rbrace, \left\lbrace6, 9\right\rbrace, \left\lbrace7, 10\right\rbrace, \left\lbrace11, 12\right\rbrace, \left\lbrace13, 16\right\rbrace, \left\lbrace14, 17\right\rbrace, \left\lbrace15, 18\right\rbrace, \left\lbrace19, 20\right\rbrace\right\rbrace \\
        Y & = \left\lbrace\left\lbrace0, 3\right\rbrace, \left\lbrace1, 2\right\rbrace, \left\lbrace4, 5\right\rbrace, \left\lbrace6, 7\right\rbrace, \left\lbrace8, 9\right\rbrace, \left\lbrace10, 13\right\rbrace, \left\lbrace11, 14\right\rbrace, \left\lbrace12, 15\right\rbrace, \left\lbrace16, 19\right\rbrace, \left\lbrace17, 20\right\rbrace, \left\lbrace18, 21\right\rbrace\right\rbrace
    \end{align*}
    It is straightforward to check that \(\mathcal{F}_\gamma\) satisfies Condition C. It is also straightforward to check that
    \begin{multicols}{2}
    \begin{itemize}
        \item \(1\) lies on a \(18\)-cycle in \(R\cup B\),
        \item \(2\) lies on a \(16\)-cycle in \(B\cup G\),
        \item \(3\) lies on a \(22\)-cycle in \(G\cup Y\),
    \end{itemize}
    \columnbreak
    \begin{itemize}
        \item \(1\) and \(2\) lie on a \(8\)-cycle in \(R\cup G\),
        \item \(2\) and \(3\) lie on a \(14\)-cycle in \(B\cup Y\),
        \item and \(0\) lies on a \(14\)-cycle in \(R\cup Y\).
    \end{itemize}
\end{multicols}
   Thus, by Corollary \ref{Cor: Types after repeated extension}, \(\mathcal{F}_\beta^k\) is a 1-factorisation of \(Circ(20+4k,\left\lbrace 1,3 \right\rbrace)\) with the following types:
   \begin{multicols}{2}
     \begin{itemize}
        \item \([18+4k,4]\),
        \item \([16+4k,6]\),
        \item \([22+4k]\),
    \end{itemize}
    \columnbreak
    \begin{itemize}
        \item \([8+4k,8,6]\),
        \item \([14+4k,8]\),
        \item and \([14+4k,4,4]\).
    \end{itemize}
\end{multicols}

    As these types are clearly distinct for every integer \(k\), \(\mathcal{F}_\beta^k\) is a 6-B1F of \(Circ(22+4k,\left\lbrace 1,3 \right\rbrace)\) with types \[\left(\left[22+4k\right],\left[18+4k,4\right],\left[16+4k,6\right],\left[14+4k,8\right],\left[14+4k,4,4\right],\left[8+4k,8,6\right]\right)\] for all \(k \geq 0 \).

    Thus, we have shown that \(Circ(2n,\left\lbrace 1,3 \right\rbrace)\) admits a 6-B1F if and only if \(n \geq 9 \).
\end{proof}

\subsection{A More General Construction of B1Fs of \(\mathbf{Circ(2n,\{1,b\} )}\)}
\label{Subsubsec: A More General Construction}
Here we give constructions of 2-B1Fs of \(Circ(2n,\{1,\ell\})\) and \(Circ(2n,\{1,2\ell\})\) for certain values of \(\ell\). We note that the construction in Lemma \ref{Lem: 2-B1Fs of Circ(n,{1,2L})} is a generalisation of the construction used in Lemma \ref{Lem: Construction of 2-B1F of Circ(6a,{1,3})}. The following two constructions share a common approach. They both start with a circulant graph on \(3\ell a\) vertices. We then define the mapping \(\phi: V \rightarrow V\) where \(\phi(v) = v+\ell\) (addition perfomed modulo \(3\ell a\)).

\begin{lemma}\label{Lem: 2-B1Fs of Circ(n,{1,L})}
    Let \(\ell\) be a positive even integer and let \(n\) be an integer such that \(n>3\).\\ If \(2n \equiv 0 \pmod {3\ell}\), then \(Circ(2n,\{1,\ell\})\) admits a 2-B1F.
\end{lemma}

\begin{proof}
    Let \(G=Circ(2n,\{1,\ell\})\) with vertex set \(V=\left\lbrace 0,1, \dots, {3\ell a-1}\right\rbrace \) for some positive integer \(a\), and let the 1-factor \(F_1\) be defined as:
    \begin{align*}
        F_1	\,  = & \left\lbrace
        \left\lbrace
        x,x+\ell
        \right\rbrace
        : x \equiv 0,1,2,\dots, \ell-1 \pmod {3\ell}
        \right\rbrace            \\
        \quad \cup
                  & \left\lbrace
        \left\lbrace
        x,x+1
        \right\rbrace
        : x \equiv 2\ell, 2\ell+2, 2\ell+4, \dots, 3\ell -2 \pmod{3\ell}
        \right\rbrace.
    \end{align*}
    Further let \(F_2=\phi(F_1)\) and \(F_3=\phi(F_2)\) (noting that \(\phi(F_3)=F_1\)). Observe that the edges of \(G\) that do not appear in \(F_1, F_2,F_3\) are the edges \(\{x,x+1\}\) where \(x \equiv 1 \pmod 2\), form a 1-factor, and shall be called \(F_4\). Clearly, \(\mathcal{F}=\{F_1,F_2,F_3,F_4\}\) is a 1-factorisation of \(G\) and it follows that \(F_1\cup F_2\cong F_1\cup F_3\cong F_2\cup F_3\) and \(F_1 \cup F_4 \cong F_2\cup F_4 \cong F_3\cup F_4\). Thus, to show that \(\mathcal{F}\) is a 2-B1F, it suffices to determine the types of \(F_1 \cup F_4\) and \(F_1\cup F_2\).
    
    As mentioned previously, to show that \(\mathcal{F}\) is a 2-B1F we must determine the types of \(F_1 \cup F_4\) and \(F_1\cup F_2\).

    First, consider the union \(F_1 \cup F_4\). Let the vertex set of \(G\) be \[V=\left\lbrace 0,1, \dots, {3\ell a-1}\right\rbrace = V_0 \cup V_1 \cup \dots \cup V_{a-1} \] where \(V_j= \left\lbrace {j(3\ell)},{j(3\ell)+1}, \dots ,{(j+1)(3\ell)-1}  \right\rbrace\) for \(0 \leq j \leq a-1\). Clearly the structure of \(F_1 \cup F_4\) will be the same on each \(V_j\). Observe that \(2\ell -4\) vertices of \(V_{j}\) will lie on the 4-cycles \[({j(3\ell)+i},{j(3\ell)+i+\ell}, {j(3\ell)+i+\ell+1},{j(3\ell)+i+1})\] for \(i \in \left\lbrace 1,3, \dots, \ell -3 \right\rbrace\). The remaining \(\ell+4\) vertices of \(V_j\) will lie on the path
    \begin{align*} P_j & =[j(3\ell),j(3\ell)+\ell,j(3\ell)+\ell-1,j(3\ell)+2\ell-1] \\
        &\qquad \cup [j(3\ell)+2\ell-1, j(3\ell)+2\ell,j(3\ell)+2\ell+1, \dots, j(3\ell)+3\ell-2, j(3\ell)+3\ell-1 ]
    \end{align*}
    from \({j(3\ell)}\) to \({(j+1)(3\ell) -1}\) (See Figure \ref{Fig: P_0} for an example). When we consider the union of all \(V_j\) it is easy to see that the paths \(P_j\) along with the edges of the form \(\left\lbrace j(3\ell)-1,j(3\ell) \right\rbrace\) from \(F_4\) will form the \(a(\ell+4)\)-cycle \[P_0\cup \left\lbrace (3\ell) -1,(3\ell) \right\rbrace \cup P_1 \cup \left\lbrace 2(3\ell) -1,2(3\ell) \right\rbrace \cup P_3 \cup \dots \cup P_a \cup \left\lbrace a(3\ell) -1,0 \right\rbrace.\]
    Thus, \(F_1\cup F_4, F_2\cup F_4\), and \(F_3\cup F_4\) will have type \(\left[a(\ell+4),4^{a(\tfrac{\ell}{2}-1)}\right]\).

    \begin{figure}[ht!]
        \begin{center}
            \begin{tikzpicture}[dot/.style={circle, fill=black, inner sep=0pt, minimum size=1pt},ddot/.style={circle,fill=black,inner sep=0pt,minimum size=10pt},dddot/.style={circle,fill=black,inner sep=0pt,minimum size=2pt},
                    lbl/.style={}]
                \def\n{40} 
                \def\f{15} 
                \def\lf{\the\numexpr\f} 
                \def\b{0} 
                \def\lb{4} 

                \def\radius{6cm} 
                \foreach \s in {0}
                {
                \node[dot,label={[lbl,label distance=2mm]360/\n * (-\s)+158:\(0\)}] (\s) at ({360/\n * (-\s)+158}:\radius) {};
                }
                \foreach \s in {4}
                {
                \node[dot,label={[lbl,label distance=2mm]360/\n * (-\s)+158:\(\ell-1\)}] (\s) at ({360/\n * (-\s)+158}:\radius) {};
                }

                \foreach \s in {5}
                    {
                        \node[dot,label={[lbl,label distance=2mm]360/\n * (-\s)+158:\(\ell\)}] (\s) at ({360/\n * (-\s)+158}:\radius) {};
                    }
                \foreach \s in {9}
                    {
                        \node[dot,label={[lbl,label distance=2mm]360/\n * (-\s)+158:\(2\ell-1\)}] (\s) at ({360/\n * (-\s)+158}:\radius) {};
                    }
                \foreach \s in {10}
                    {
                        \node[dot,label={[lbl,label distance=2mm]360/\n * (-\s)+158:\(2\ell\)}] (\s) at ({360/\n * (-\s)+158}:\radius) {};
                    }
                    
                \foreach \s in {11}
                {
                    \node[dot,label={[lbl,label distance=2mm]360/\n * (-\s)+158:\(2\ell+\the\numexpr\s-10\)}] (\s) at ({360/\n * (-\s)+158}:\radius) {};
                }
                \foreach \s in {13,14}
                {
                    \node[dot,label={[lbl,label distance=2mm]360/\n * (-\s)+158:\(3\ell-\the\numexpr 15-\s\)}] (\s) at ({360/\n * (-\s)+158}:\radius) {};
                }
                \foreach \s in {15}
                {
                    \node[dot,label={[lbl,label distance=2mm]360/\n * (-\s)+158:\(3\ell\)}] (\s) at ({360/\n * (-\s)+158}:\radius) {};
                }
                \foreach \s in {\lf,...,\numexpr\f\relax}
                    {
                        \node[dot] (\s) at ({360/\n * (-\s)+158}:\radius) {};
                    }
                
                \foreach \s in {0,...,1}
                    {
                        \node[ddot] (disp\s) at ({360/\n * (-\s)+158}:\radius) {};
                    }
                \foreach \s in {-1,...,1}
                    {
                        \node[dddot] (disp2\s) at ({360/\n * (-2)+360/(\n*10) * (-\s)+158}:\radius) {};
                    }    
                \foreach \s in {3,...,6}
                    {
                        \node[ddot] (disp\s) at ({360/\n * (-\s)+158}:\radius) {};
                    }
                \foreach \s in {-1,...,1}
                {
                    \node[dddot] (disp7\s) at ({360/\n * (-7)+360/(\n*10) * (-\s)+158}:\radius) {};
                }    
                \foreach \s in {8,...,11}
                    {
                        \node[ddot] (disp\s) at ({360/\n * (-\s)+158}:\radius) {};
                    }
                    \foreach \s in {-1,...,1}
                    {
                        \node[dddot] (disp12\s) at ({360/\n * (-12)+360/(\n*10) * (-\s)+158}:\radius) {};
                    }   
                \foreach \s in {13,...,\numexpr\f\relax}
                    {
                        \node[ddot] (disp\s) at ({360/\n * (-\s)+158}:\radius) {};
                    }
                \begin{scope}[on background layer]
                    \foreach \s/\y in  {4/5, 9/10, 11/disp12-1,disp121/13, 14/15}
                        {
                       \draw [circ_darkred]  (\s) -- (\y);
                        }
                \foreach \s/\y in {0/5,4/9}
                    {
                        \draw [circ_darkblue]  (\s) to[bend left=-10] (\y);
                    }
                \foreach \s/\y in {10/11,13/14}
                    {
                        \draw [circ_darkblue]  (\s) -- (\y);
                    }
                    
                \end{scope}
                \begin{scope}[node distance=1cm, every node/.style={font=\sffamily}, align=left]
                    \matrix [above right=of current bounding box.north east, yshift=-0.75cm,xshift=-0.25cm,anchor=north east, nodes={inner sep=0pt}, row sep=0.15cm] {
                    \node [label=right:{\(F_1\)}] (legend1) {};
                    \draw [circ_darkblue] ([xshift=-1.1cm]legend1.west) -- ([xshift=-0.1cm]legend1.west);
                    \\
                    \node [label=right:{\(F_4\)}] (legend2) {};
                    \draw [circ_darkred] ([xshift=-1.1cm]legend2.west) -- ([xshift=-0.1cm]legend2.west);
                    \\
                    };
                \end{scope}
            \end{tikzpicture}
            \caption{\(P_0\)}\label{Fig: P_0}
        \end{center}
    \end{figure}

    We now consider the union \(F_1 \cup F_2\). It is easy to confirm that it is a collection of \(6\)-cycles of the form \[(i, {i+1}, {i+1+\ell},{i+1+2\ell},{i+2\ell},{i+\ell}).\] Thus, \(\mathcal{F}\) is a 2-B1F with types \(\left(\left[a(\ell+4),4^{a(\tfrac{\ell}{2}-1)}\right],\left[6^{\tfrac{\ell a}{2}}\right]\right)\).
\end{proof}

\begin{lemma} \label{Lem: 2-B1Fs of Circ(n,{1,2L})}
    Let \(\ell\) be a positive even integer and let \(n\) be an integer such that \(n>3\).\\ If \(2n \equiv 0 \pmod {3\ell}\), then \(Circ(2n,\{1,2\ell\})\) admits a 2-B1F.
\end{lemma}

\begin{proof}
    Let \(G=Circ(2n,\{1,{2\ell}\})\) with vertex set \(V=\left\lbrace 0,1, \dots, {3\ell a-1}\right\rbrace \) for some positive integer \(a\), and let the 1-factor \(F_1\) be defined as:
    \begin{align*}
        F_1	\,  = & \left\lbrace
        \left\lbrace
        x,x+2\ell
        \right\rbrace
        : x \equiv 0,1,2,\dots, \ell-1 \pmod {3\ell}
        \right\rbrace            \\
        \quad \cup
                  & \left\lbrace
        \left\lbrace
        x,x+1
        \right\rbrace
        : x \equiv \ell, \ell+2, \ell+4, \dots, 2\ell -2 \pmod{3\ell}
        \right\rbrace.
    \end{align*}

    Further let \(F_2=\phi(F_1)\) and \(F_3=\phi(F_2)\) (noting that \(\phi(F_3)=F_1\)). Observe that the edges of \(G\) that do not appear in \(F_1, F_2,F_3\) are the edges \(\{x,x+1\}\) where \(x \equiv 1 \pmod 2\), form a 1-factor, and shall be called \(F_4\). Clearly, \(\mathcal{F}=\{F_1,F_2,F_3,F_4\}\) is a 1-factorisation of \(G\) and it follows that \(F_1\cup F_2\cong F_1\cup F_3\cong F_2\cup F_3\) and \(F_1 \cup F_4 \cong F_2\cup F_4 \cong F_3\cup F_4\). Thus, to show that \(\mathcal{F}\) is a 2-B1F, it suffices to determine the types of \(F_1 \cup F_4\) and \(F_1\cup F_2\).

    First, consider the union \(F_1 \cup F_4\). Let the vertex set of \(G\) be \[V=\left\lbrace 0,1, \dots, {3\ell a-1}\right\rbrace = V_0 \cup V_1 \cup \dots \cup V_{a-1} \] where \(V_j= \left\lbrace {j(3\ell)}, {j(3\ell)+1}, \dots ,{(j+1)(3\ell)-1}  \right\rbrace\) for \(0 \leq j \leq a-1\). Clearly the structure of \(F_1 \cup F_4\) will be the same on each \(V_j\). Observe that \(2\ell -4\) vertices of \(V_{j}\) will lie on the 4-cycles \[({j(3\ell)+i},{j(3\ell)+i+2\ell}, {j(3\ell)+i+2\ell+1},{j(3\ell)+i+1})\] for \(i \in \left\lbrace 1,3, \dots, \ell -3 \right\rbrace\). The remaining \(\ell+4\) vertices of \(V_j\) will lie on the path
    \begin{align*} Q_j &=[j(3\ell),j(3\ell)+2\ell] \\ &\quad \cup [j(3\ell)+2\ell ,j(3\ell)+2\ell-1,j(3\ell)+2\ell-2, \dots ,j(3\ell)+\ell, j(3\ell)+\ell-1] \\ &\quad \cup [j(3\ell)+\ell-1, j(3\ell)+3\ell-1]
    \end{align*}
    from \({j(3\ell)}\) to \({(j+1)(3\ell) -1}\) (See Figure \ref{Fig: Q_0} for an example). When we consider the union of all \(V_j\) it is easy to see that the paths \(Q_j\) along with the edges of the form \(\left\lbrace j(3\ell)-1,j(3\ell) \right\rbrace\) from \(F_4\) will form the \(a(\ell+4)\)-cycle \[Q_0\cup \left\lbrace (3\ell) -1,(3\ell) \right\rbrace \cup Q_1 \cup \left\lbrace 2(3\ell) -1,2(3\ell) \right\rbrace \cup Q_3 \cup \dots \cup Q_a \cup \left\lbrace a(3\ell) -1,0 \right\rbrace.\]
    Thus \(F_1\cup F_4, F_2\cup F_4\), and \(F_3\cup F_4\) will have type \(\left[a(\ell+4),4^{a(\tfrac{\ell}{2}-1)}\right]\). 
    \begin{figure}[ht!]
        \begin{center}
            \begin{tikzpicture}[dot/.style={circle, fill=black, inner sep=0pt, minimum size=1pt},ddot/.style={circle,fill=black,inner sep=0pt,minimum size=10pt},dddot/.style={circle,fill=black,inner sep=0pt,minimum size=2pt},
                    lbl/.style={}]
                \def\n{40} 
                \def\f{15} 
                \def\lf{\the\numexpr\f} 
                \def\b{0} 
                \def\lb{4} 
    
                \def\radius{6cm} 
                \foreach \s in {0}
                {
                \node[dot,label={[lbl,label distance=2mm]360/\n * (-\s)+158:\(0\)}] (\s) at ({360/\n * (-\s)+158}:\radius) {};
                }
                \foreach \s in {4}
                {
                \node[dot,label={[lbl,label distance=2mm]360/\n * (-\s)+158:\(\ell-1\)}] (\s) at ({360/\n * (-\s)+158}:\radius) {};
                }
    
                \foreach \s in {5}
                    {
                        \node[dot,label={[lbl,label distance=2mm]360/\n * (-\s)+158:\(\ell\)}] (\s) at ({360/\n * (-\s)+158}:\radius) {};
                    }
                    \foreach \s in {6}
                    {
                        \node[dot,label={[lbl,label distance=2mm]360/\n * (-\s)+158:\(\ell+1\)}] (\s) at ({360/\n * (-\s)+158}:\radius) {};
                    }
                    \foreach \s in {7}
                    {
                        \node[dot] (\s) at ({360/\n * (-\s)+158}:\radius) {};
                    }
    
                \foreach \s in {8,9}
                    {
                        \node[dot,label={[lbl,label distance=2mm]360/\n * (-\s)+158:\(2\ell-\the\numexpr 10-\s\)}] (\s) at ({360/\n * (-\s)+158}:\radius) {};
                    }
                \foreach \s in {10}
                    {
                        \node[dot,label={[lbl,label distance=2mm]360/\n * (-\s)+158:\(2\ell\)}] (\s) at ({360/\n * (-\s)+158}:\radius) {};
                    }
                    
                \foreach \s in {11,12,13}
                {
                    \node[dot] (\s) at ({360/\n * (-\s)+158}:\radius) {};
                }
                \foreach \s in {14}
                {
                    \node[dot,label={[lbl,label distance=2mm]360/\n * (-\s)+158:\(3\ell-\the\numexpr 15-\s\)}] (\s) at ({360/\n * (-\s)+158}:\radius) {};
                }
                \foreach \s in {15}
                {
                    \node[dot,label={[lbl,label distance=2mm]360/\n * (-\s)+158:\(3\ell\)}] (\s) at ({360/\n * (-\s)+158}:\radius) {};
                }
                \foreach \s in {\lf,...,\numexpr\f\relax}
                    {
                        \node[dot] (\s) at ({360/\n * (-\s)+158}:\radius) {};
                    }
                
                \foreach \s in {0,...,1}
                    {
                        \node[ddot] (disp\s) at ({360/\n * (-\s)+158}:\radius) {};
                    }
                \foreach \s in {-1,...,1}
                    {
                        \node[dddot] (disp2\s) at ({360/\n * (-2)+360/(\n*10) * (-\s)+158}:\radius) {};
                    }    
                \foreach \s in {3,...,6}
                    {
                        \node[ddot] (disp\s) at ({360/\n * (-\s)+158}:\radius) {};
                    }
                \foreach \s in {-1,...,1}
                {
                    \node[dddot] (disp7\s) at ({360/\n * (-7)+360/(\n*10) * (-\s)+158}:\radius) {};
                }    
                \foreach \s in {8,...,11}
                    {
                        \node[ddot] (disp\s) at ({360/\n * (-\s)+158}:\radius) {};
                    }
                    \foreach \s in {-1,...,1}
                    {
                        \node[dddot] (disp12\s) at ({360/\n * (-12)+360/(\n*10) * (-\s)+158}:\radius) {};
                    }   
                \foreach \s in {13,...,\numexpr\f\relax}
                    {
                        \node[ddot] (disp\s) at ({360/\n * (-\s)+158}:\radius) {};
                    }
                \begin{scope}[on background layer]
                    \foreach \s/\y in  {4/5, 9/10,14/15,6/disp7-1, disp71/8}
                        {
                       \draw [circ_darkred]  (\s) -- (\y);
                        }
                \foreach \s/\y in {0/10,5/6,8/9,4/14}
                    {
                        \draw [circ_darkblue]  (\s) -- (\y);
                    }
                    
                \end{scope}
                \begin{scope}[node distance=1cm, every node/.style={font=\sffamily}, align=left]
                    \matrix [above right=of current bounding box.north east, yshift=-0.75cm,xshift=-0.25cm,anchor=north east, nodes={inner sep=0pt}, row sep=0.15cm] {
                    \node [label=right:{\(F_1\)}] (legend1) {};
                    \draw [circ_darkblue] ([xshift=-1.1cm]legend1.west) -- ([xshift=-0.1cm]legend1.west);
                    \\
                    \node [label=right:{\(F_4\)}] (legend2) {};
                    \draw [circ_darkred] ([xshift=-1.1cm]legend2.west) -- ([xshift=-0.1cm]legend2.west);
                    \\
                    };
                \end{scope}
            \end{tikzpicture}
            \caption{\(Q_0\)}\label{Fig: Q_0}
        \end{center}
    \end{figure}
    
    We now consider the union \(F_1 \cup F_2\). It is easy to confirm that it is a collection of \(6\)-cycles of the form \[(i, {i+1}, {i+1+2\ell},{i+1+4\ell},{i+4\ell},{i+2\ell}).\] Thus, \(\mathcal{F}\) is a 2-B1F with types \(\left(\left[a(\ell+4),4^{a(\tfrac{\ell}{2}-1)}\right],\left[6^{\tfrac{\ell a}{2}}\right]\right)\).
\end{proof}

\subsection{Computational Results}\label{Subsubsec: Computational Results}
Note in Theorems \ref{Thm: 2-B1Fs of Circ(n,{1,3})}, \ref{Thm: 3-B1Fs of Circ(n,{1,3})}, and \ref{Thm: 6-B1Fs of Circ(n,{1,3})} we referred to computational results. In this section we describe how they were achieved, as well as giving some additional computational results.

By modifying the algorithm used by Herke \cite{SaraPhDThesis} to enumerate all P1Fs of \(Circ(2n,\left\lbrace a,b \right\rbrace)\), we were able to enumerate all 1-factorisations of connected circulants \(Circ(2n,\left\lbrace a,b \right\rbrace)\) for small values of \(n,a,\) and \(b\) and determine for which of these values \(m\)-B1Fs exist for \(m\in \left\lbrace 2,3,6 \right\rbrace\). It is known that two 3- or 4-regular circulant graphs \(Circ(2n,\left\lbrace d_1,d_2 \right\rbrace)\) and \(Circ(2n,\left\lbrace d_1',d_2' \right\rbrace)\) are isomorphic if and only if \(\left\lbrace md_1',md_2' \right\rbrace = \left\lbrace \pm d_1, \pm d_2 \right\rbrace\) modulo \(2n\) for some \(m\in \mathbb{Z}^*_{2n}\) (see \cite{Herke2013P1FCirc} for example).
The algorithm relies on this fact, as it allows us to determine all non-isomorphic 4-regular circulant graphs by their connection sets.
\begin{table}[h]
    \centering
    \small
    \begin{longtable}{|c|c|c|c|c|}
        \hline
        \(n\)                                                           & 2-B1FS &3-B1FS & 6-B1FS &No \(m\)-B1Fs \\ \hline
        3                                                               &                                                         &                                                         &  &                                          \\ \hline
        4                                                               &
        \begin{tabular}{c} \(\begin{aligned}
                &\left\lbrace 1,2 \right\rbrace
            \end{aligned}\)\end{tabular} &
        \begin{tabular}{c}\end{tabular}                                 &
        \begin{tabular}{c}\end{tabular}                    &
        \begin{tabular}{c} \(\begin{aligned}
            &\left\lbrace 1,3 \right\rbrace
        \end{aligned}\)\end{tabular}                                                                                                                                                                                            \\ \hline
        5                                                              &
        \begin{tabular}{c}\(\begin{aligned}
                &\left\lbrace 1,3 \right\rbrace
            \end{aligned}\)\end{tabular}  &
        \begin{tabular}{c}\end{tabular}                                 &
        \begin{tabular}{c}\end{tabular}                                                                                                                                                                                                        &   \begin{tabular}{c} \(\begin{aligned}
            &\left\lbrace 1,2 \right\rbrace, \left\lbrace 1,4 \right\rbrace
        \end{aligned}\)\end{tabular}      \\ \hline
        6                                                              &
        \begin{tabular}{c}\(\begin{aligned}
                 & \left\lbrace 1,2 \right\rbrace,\left\lbrace 1,3 \right\rbrace, \\
                 &\left\lbrace 1,4 \right\rbrace, \left\lbrace 2,3 \right\rbrace, \\
                 &\left\lbrace 3,4 \right\rbrace
            \end{aligned}\)\end{tabular}  &
        \begin{tabular}{c}\(\begin{aligned}
                &\left\lbrace 1,3 \right\rbrace,\left\lbrace 1,5 \right\rbrace
            \end{aligned}\)\end{tabular}  &
        \begin{tabular}{c}\end{tabular}                                                                                                                                                                                                               & \begin{tabular}{c} \(\begin{aligned}
        \end{aligned}\)\end{tabular}   \\ \hline
        7                                                              &
        \begin{tabular}{c}\(\begin{aligned}
                &\left\lbrace 1,3 \right\rbrace
            \end{aligned}\)\end{tabular}  &
        \begin{tabular}{c}\(\begin{aligned}
                &\left\lbrace 1,3 \right\rbrace,\left\lbrace 1,4 \right\rbrace
            \end{aligned}\)\end{tabular}  &
        \begin{tabular}{c}\end{tabular}                                                                                                                                                                                                             & \begin{tabular}{c} \(\begin{aligned}
            &\left\lbrace 1,2 \right\rbrace, \left\lbrace 1,6 \right\rbrace
        \end{aligned}\)\end{tabular}   \\ \hline
        8                                                              &
        \begin{tabular}{c}\(\begin{aligned}
                &\left\lbrace 1,3 \right\rbrace,\left\lbrace 1,4 \right\rbrace
            \end{aligned}\)\end{tabular}  &
        \begin{tabular}{c}\(\begin{aligned}
                &\left\lbrace 1,3 \right\rbrace,\left\lbrace 1,4 \right\rbrace, \\
                &\left\lbrace 1,7 \right\rbrace
            \end{aligned}\)\end{tabular}  &
        \begin{tabular}{c}\end{tabular}                                                                                                                                                                                                             & \begin{tabular}{c} \(\begin{aligned}
            &\left\lbrace 1,2 \right\rbrace, \left\lbrace 1,6 \right\rbrace
        \end{aligned}\)\end{tabular}   \\ \hline
        9                                                              &
        \begin{tabular}{c}\(\begin{aligned}
                 & \left\lbrace 1,2 \right\rbrace,\left\lbrace 1,3 \right\rbrace,\\
                 &\left\lbrace 1,4 \right\rbrace, \left\lbrace 1,5 \right\rbrace,\\
                 &\left\lbrace 1,6 \right\rbrace,\left\lbrace 2,3 \right\rbrace
            \end{aligned}\)\end{tabular}  &
        \begin{tabular}{c}\(\begin{aligned}
                 & \left\lbrace 1,3 \right\rbrace,\left\lbrace 1,4 \right\rbrace,\\
                 &\left\lbrace 1,5 \right\rbrace, \left\lbrace 1,6 \right\rbrace,\\
                 &\left\lbrace 1,8 \right\rbrace,\left\lbrace 2,3 \right\rbrace
            \end{aligned}\)\end{tabular}  &
        \begin{tabular}{c}\(\begin{aligned}
                &\left\lbrace 1,3 \right\rbrace,\left\lbrace 1,4 \right\rbrace,\\
                &\left\lbrace 1,5 \right\rbrace
            \end{aligned}\)\end{tabular}                                                                                                                                                                              & \begin{tabular}{c} \(\begin{aligned}
                &\left\lbrace 1,2 \right\rbrace, \left\lbrace 1,3 \right\rbrace
            \end{aligned}\)\end{tabular}   \\ \hline
        10                                                              &
        \begin{tabular}{c}\(\begin{aligned}
                 & \left\lbrace 1,3 \right\rbrace,\left\lbrace 1,4 \right\rbrace,\\
                 &\left\lbrace 1,5 \right\rbrace, \left\lbrace 1,6 \right\rbrace,\\
                 &\left\lbrace 2,5 \right\rbrace,\left\lbrace 4,5 \right\rbrace
            \end{aligned}\)\end{tabular}  &
        \begin{tabular}{c}\(\begin{aligned}
                 & \left\lbrace 1,3 \right\rbrace,\left\lbrace 1,4 \right\rbrace,\\
                 &\left\lbrace 1,5 \right\rbrace, \left\lbrace 1,6 \right\rbrace,\\
                 &\left\lbrace 1,8 \right\rbrace,\left\lbrace 1,9 \right\rbrace
            \end{aligned}\)\end{tabular}  &
        \begin{tabular}{c}\(\begin{aligned}
                 & \left\lbrace 1,3 \right\rbrace,\left\lbrace 1,4 \right\rbrace,\\
                 &\left\lbrace 1,5 \right\rbrace, \left\lbrace 1,6 \right\rbrace,\\
                 & \left\lbrace 1,8 \right\rbrace,\left\lbrace 2,5 \right\rbrace, \\
                 & \left\lbrace 4,5 \right\rbrace
            \end{aligned}\)\end{tabular}                                                                                                                                                                     & \begin{tabular}{c} \(\begin{aligned}
                &\left\lbrace 1,2 \right\rbrace
            \end{aligned}\)\end{tabular}             \\ \hline
        11                                                              &
        \begin{tabular}{c}\(\begin{aligned}
                &\left\lbrace 1,3 \right\rbrace,\left\lbrace 1,5 \right\rbrace,\\
                &\left\lbrace 1,6 \right\rbrace
            \end{aligned}\)\end{tabular}  &
        \begin{tabular}{c}\(\begin{aligned}
                 & \left\lbrace 1,3 \right\rbrace,\left\lbrace 1,4 \right\rbrace,\\
                 &\left\lbrace 1,5 \right\rbrace, \left\lbrace 1,6 \right\rbrace,\\
                 &\left\lbrace 1,8 \right\rbrace,\left\lbrace 1,10 \right\rbrace
            \end{aligned}\)\end{tabular}  &
        \begin{tabular}{c}\(\begin{aligned}
                 & \left\lbrace 1,3 \right\rbrace,\left\lbrace 1,4 \right\rbrace,\\
                 &\left\lbrace 1,5 \right\rbrace, \left\lbrace 1,6 \right\rbrace,\\
                 &\left\lbrace 1,8 \right\rbrace
            \end{aligned}\)\end{tabular}                                                                                                                                                                      & \begin{tabular}{c} \(\begin{aligned}
                &\left\lbrace 1,2 \right\rbrace
            \end{aligned}\)\end{tabular}           \\ \hline
        12                                                              &
        \begin{tabular}{c}\(\begin{aligned}
                 & \left\lbrace 1,2 \right\rbrace,\left\lbrace 1,3 \right\rbrace,\\
                 &\left\lbrace 1,4 \right\rbrace, \left\lbrace 1,5 \right\rbrace, \\
                 & \left\lbrace 1,6 \right\rbrace,\left\lbrace 1,7 \right\rbrace, \\
                 & \left\lbrace 1,8 \right\rbrace,\left\lbrace 1,9 \right\rbrace,\\
                 & \left\lbrace 1,10\right\rbrace,\left\lbrace 2,3 \right\rbrace,\\
                 &\left\lbrace 3,4 \right\rbrace,\left\lbrace 3,8 \right\rbrace
            \end{aligned}\)\end{tabular}  &
        \begin{tabular}{c}\(\begin{aligned}
                 & \left\lbrace 1,3 \right\rbrace,\left\lbrace 1,4 \right\rbrace,\\
                 &\left\lbrace 1,5 \right\rbrace, \left\lbrace 1,7 \right\rbrace,\\
                 &\left\lbrace 1,9 \right\rbrace,\left\lbrace 1,10 \right\rbrace, \\
                 & \left\lbrace 1,11 \right\rbrace,\left\lbrace 2,3 \right\rbrace,\\
                 &\left\lbrace 3,4 \right\rbrace
            \end{aligned}\)\end{tabular}  &
        \begin{tabular}{c}\(\begin{aligned}
                 & \left\lbrace 1,3 \right\rbrace,\left\lbrace 1,4 \right\rbrace,\\
                 &\left\lbrace 1,5 \right\rbrace, \left\lbrace 1,6 \right\rbrace, \\
                 &\left\lbrace 1,7 \right\rbrace,\left\lbrace 1,8 \right\rbrace,  \\
                 & \left\lbrace 1,9 \right\rbrace,\left\lbrace 1,10 \right\rbrace,\\
                 & \left\lbrace 2,3 \right\rbrace, \left\lbrace 3,4 \right\rbrace,\\
                 & \left\lbrace 3,8 \right\rbrace
            \end{aligned}\)\end{tabular}                                                                                                                                                             &   \begin{tabular}{c} \(\begin{aligned}

            \end{aligned}\)\end{tabular}                  \\ \hline
    \end{longtable}
    \caption{Connection sets \(\left\lbrace a,b \right\rbrace\) and the existence or non-existence of \(m\)-B1Fs of \(Circ(2n,\left\lbrace a,b \right\rbrace)\) for \(n\leq 12\).}
    \label{Tab: m-B1F existence for connected 4-regular circulant graphs on at most 24 vertices}
\end{table}
\newpage
\bibliographystyle{abbrv}

\end{document}